\documentclass[11pt, reqno]{amsart}

\usepackage{amsmath,amsfonts,amsthm,amssymb,color,hyperref, mathrsfs, graphicx, mathabx,tikz}

\usepackage[T1]{fontenc}

\usepackage[left=1.1in, right=1.1in, top=1.1in,bottom=1.1in]{geometry} 

\newtheorem{theorem}{Theorem}[section]
\newtheorem{lemma}[theorem]{Lemma}

\theoremstyle{definition}

\theoremstyle{remark}
\newtheorem{remark}{Remark}
\numberwithin{equation}{section}

\newcommand{\C}{\mathbb C}

\newcommand{\R}{\mathbb R}
 \newcommand{\N}{\mathbb N}

\newcommand{\Z}{\mathbb Z}
\newcommand{\E}{\mathbb E}
 
 \newcommand{\cF}{\mathcal{F}}
\newcommand{\fH}{\mathfrak{H}}
 \newcommand{\PP}{\mathbb P}
  
    \newcommand{\cO}{\mathcal O}

\newcommand{\1}{\textbf{1}}
\newcommand{\kk}{\kappa}

\newcommand{\e}{{\varepsilon}}

\newcommand{\var}{\textup{Var}}
\newcommand{\cov}{\textup{Cov}}
\newcommand{\absolute}[1]{\vert{#1}\vert}

\newcommand{\fixGraph}{\Gamma}
\newcommand{\vertices}[1]{\upsilon_{#1}}
\newcommand{\edges}[1]{\varepsilon_{#1}}

\begin{document}

\title[A discrete  second-order Gaussian Poincar\'e inequality with applications]{A simplified second-order Gaussian Poincar\'e inequality\\ in discrete setting with applications}

 \date{}

 \author[P.\ Eichelsbacher]{Peter Eichelsbacher} 
 \address{Faculty of Mathematics, Ruhr University Bochum, Germany.}
 \email{peter.eichelsbacher@rub.de}

 \author[B.\ Redno\ss]{Benedikt Redno\ss} 
 \address{Faculty of Mathematics, Ruhr University Bochum, Germany.}
 \email{benedikt.rednoss@rub.de}
  
 \author[C.\ Th\"ale]{Christoph Th\"ale} 
 \address{Faculty of Mathematics, Ruhr University Bochum, Germany.}
 \email{christoph.thaele@rub.de}

   
\author[G.\ Zheng]{Guangqu Zheng}  
\address{School of Mathematics, The University of Edinburgh, UK.}
\email{zhengguangqu@gmail.com}

\maketitle


\begin{abstract}  In this paper, a simplified second-order Gaussian Poincar\'e inequality for normal approximation of functionals over infinitely many  Rademacher random variables is derived. It is based on a new bound for the Kolmogorov distance between a general Rademacher functional and a Gaussian random variable, which is established by means of the discrete Malliavin-Stein method and is of independent interest. 
As an application, the number of vertices with prescribed degree and the subgraph counting statistic in the Erd\H{o}s-R\'enyi random graph are discussed. 
The number of vertices of fixed degree is also studied for percolation on the Hamming hypercube. Moreover, the number of isolated faces in the Linial-Meshulam-Wallach random $\kappa$-complex and infinite weighted 2-runs are treated.\\[0.5cm]
  \end{abstract}

\tableofcontents

\section{Introduction and applications}

In the last decade, a wide range of quantitative central limit type results for random systems driven by either a Gaussian or a Poisson process have been obtained by means for the so-called Malliavin-Stein method, which combines the Malliavin calculus of variations on a Gaussian space or on configuration spaces with Stein's method for normal approximation. This Malliavin-Stein approach was first developed by Nourdin and Peccati during their successful attempt in quantifying the celebrated fourth moment theorem and it has been exploited for problems related to excursion sets of Gaussian random fields, nodal statistics of random waves, random geometric graphs, random tessellations or random polytopes, to name just a few. We refer to the monographs \cite{NPbook,PeccatiReitznerBook} as well as to the webpage \cite{WWW} for an extensive overview. In parallel to these developments, the Malliavin-Stein method has also been made available for discrete random structures, which can be described by possibly infinitely many independent Rademacher random variables taking values $+1$ and $-1$ only, see \cite{KRT16,KRT17,KT,NourdinPeccatiReinert}. Applications of this discrete Malliavin-Stein technique to quantitative central limit theorems for random graphs, random simplicial complexes and percolation models have been the content of \cite{KRT17,KT}. The technical backbone of  these papers is what is known as the {discrete second-order  Gaussian   Poincar\'e inequality}, which in turn relies on an abstract normal approximation bound previously developed in \cite{KRT16}. The purpose of the present article is to provide  a significantly simplified version of such an inequality,  whose proof  is based on a new and powerful bound for the normal approximation of non-linear functionals of an infinite Rademacher sequence, which is of independent interest. It is mainly guided by a monotonicity property of the solutions of Stein's equation for normal approximation. This property has been proved and applied at many places in Stein's method, but in the proof of Theorem 2.2 in \cite{ShaoZhang19} it was identified the first time
to prove significantly simplified bounds for normal and non-normal approximations for unbounded exchangeable pairs. Moreover in \cite{Shao19} the
same monotonicity argument was applied for improvements in further approaches of Stein's method. Beyond that our text has also been inspired by the recent work \cite{LRPY21} dealing with the normal approximation of Poisson functionals.

\subsection{Application to infinite weighted $2$-runs}

In order to demonstrate the power of the new bounds  we develop in this paper, we shall in this and the next subsections describe and  discuss  a number of applications.  The first is concerned with infinite weighted $2$-runs, which have previously been analysed in \cite{KRT16,NourdinPeccatiReinert}. We remark that our bound is of the same order as the one obtained in \cite[Proposition 5.3]{NourdinPeccatiReinert} for a smooth probability metric. At the same time it simplifies the statement and the proof of the Berry-Esseen bound from \cite[Theorem 6.1]{KRT16}. We recall that for two real-valued random variables $X$ and $Y$, {defined over the same probability space $(\Omega,\mathcal{A},\PP)$}, their  Kolmogorov distance is defined as
\[
d_K(X,Y):=\sup_{z\in\R}|\PP(X\leq z)-\PP(Y\leq z)|.
\]
In what follows, we use the usual big-O notation $\cO(\,\cdot\,)$ with the meaning that the implicit constant does not depend on the parameters in brackets. Throughout this paper, we write $N\sim \mathcal{N}(0,1)$ to mean that $N$ is a standard normal random variable. Moreover, for a sequence $\alpha=\{\alpha_i:i\in\Z\}$  and $p>0$ we write $\|\alpha\|_{\ell^p(\Z)}:=(\sum_{i\in\Z}|\alpha_i|^p)^{1/p}$.

\begin{theorem}\label{thm:2runs}
	Let $\{\xi_i, i\in\Z\}$ be a sequence of independent Bernoulli random variables satisfying $\PP(\xi_i = 0) =\PP(\xi_i = 1) =1/2$ and let $\alpha^{(n)} = \{\alpha^{(n)}_i: i\in\Z\} $ be a square-summable sequence for each $n\in\N$.  We define the infinite $2$-run $G_n := \sum_{i\in \Z} \alpha^{(n)}_i \xi_i \xi_{i+1}$, $n\in\N$. 
	Then  with $N\sim\mathcal{N}(0,1)$, we have 
	\begin{align}  
		d_K\left(  \frac{G_n - \E G_n }{\sqrt{\var(G_n) }},  N  \right)  &= \cO\left( \frac{\|\alpha^{(n)}\|_{\ell^4(\Z)}^2}{\var(G_n) } \right)    =  \cO\left(    \frac{\| \alpha^{(n)} \|_{\ell^4(\Z)}^2   }{\| \alpha^{(n)} \|_{\ell^2(\Z)}^2   } \right).  \label{2run1}
	\end{align}
\end{theorem}

\subsection{Application to the Erd\H{o}s-R\'enyi random graph}  

\begin{figure}[t]
	\includegraphics[trim = 20mm 20mm 20mm 20mm, width=0.31\columnwidth]{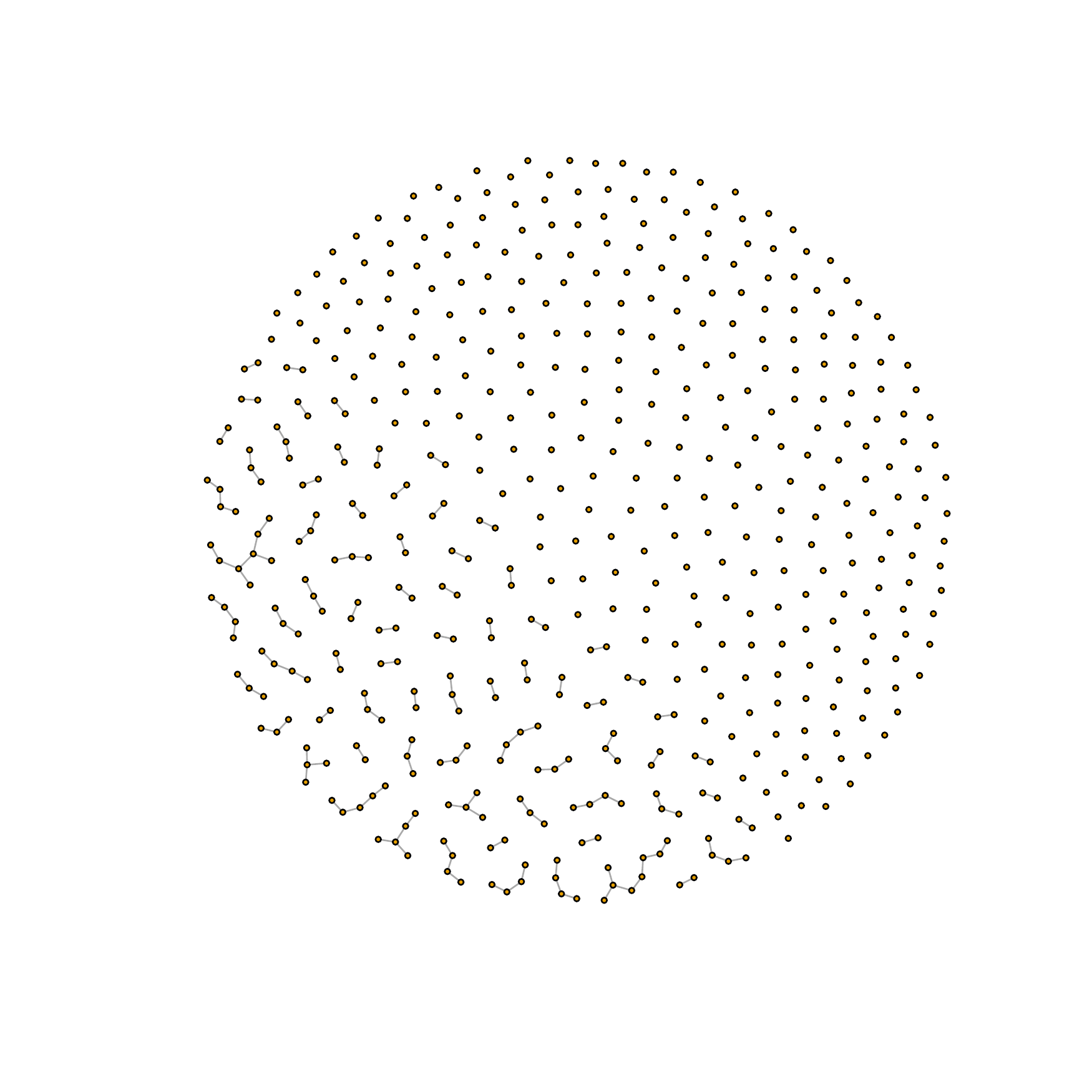}
	\includegraphics[trim = 20mm 20mm 20mm 20mm, width=0.31\columnwidth]{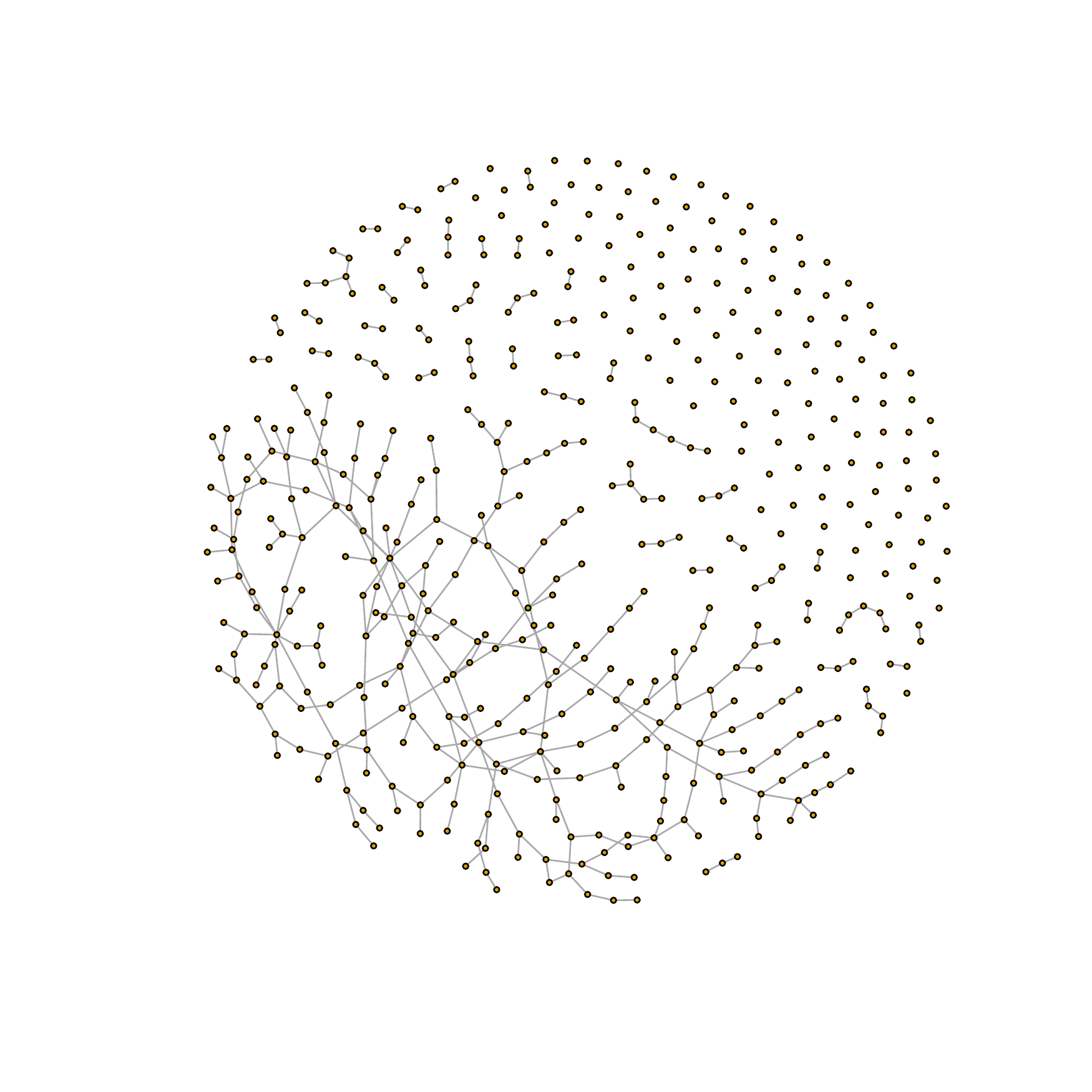}
	\includegraphics[trim = 20mm 20mm 20mm 20mm, width=0.31\columnwidth]{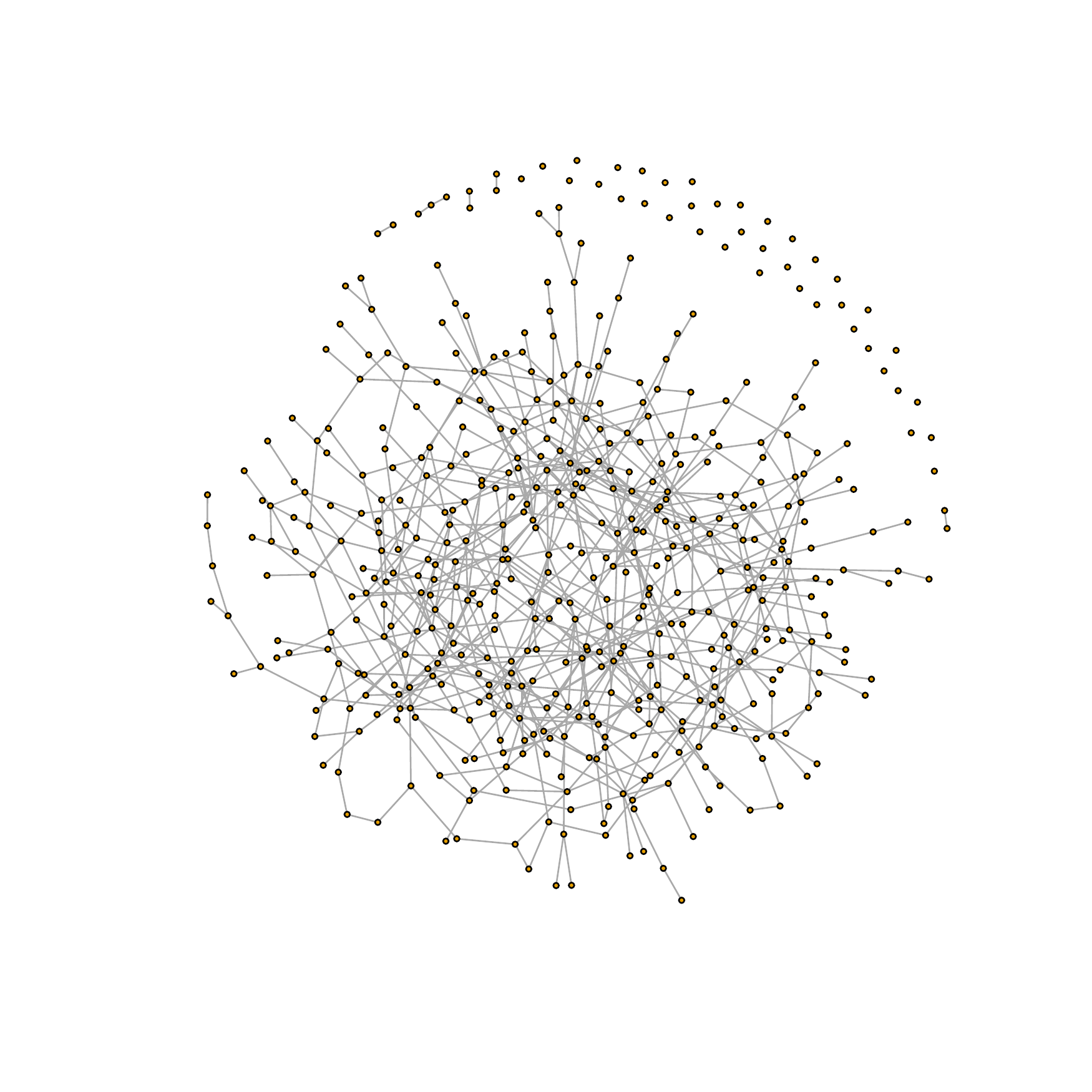}
	\caption{Simulations of the Erd\H{o}s-Renyi random graph $\mathbf{G}(n,p)$ with $n=500$ and $p=0.001$ (left panel), $p=0.003$ (middle panel) and $p=0.005$ (right panel). The graphics were produced using the \texttt{R}-package \texttt{igraph}.}
	\label{fig:ERgraph}
\end{figure}

We turn now to a first more sophisticated application concerning the classical Erd\H{o}s-R\'enyi random graph $\mathbf{G}(n,p)$.  This random graph  arises by keeping each edge of the complete graph on $n$ vertices with probability $p\in[0,1]$ and by removing it with probability $q:=1-p$, and where the decisions for the individual edges are taken independently, see Figure \ref{fig:ERgraph} for simulations. We remark that, although we shall not make this visible in our notation, we allow the probability $p$ to depend on the number of vertices $n$.

We are first interested in the number $S$ of subgraphs of $\mathbf{G}(n,p)$,
which are isomorphic\footnote{We say two graphs $G_1 = (V_1, E_1), G_2 =(V_2, E_2)$ are isomorphic if  there is a bijection $f: V_1 \to V_2$ such that any two vertices $u,v$ are adjacent in $G_1$ if and only if   $f(u)$, $f(v)$ are adjacent in $G_2$.} to a fixed graph $\fixGraph$ with at least one edge.
To study the asymptotic normality of $S$, as $n\to\infty$, we define 
\begin{align*}
W := \frac{S-\E S}{\sigma}
\end{align*}
with $\sigma := \sqrt{\var(S)}$.
Furthermore, we use the following notation.
For any graph $H$,
let $\vertices{H}$ be the number of vertices of $H$ and $\edges{H}$ its number of edges.
Finally, we define the quantity
$$
\psi = \psi(n,p,\fixGraph) := \min_H n^{\vertices{H}}p^{\edges{H}},
$$
where the minimum is taken among all subgraphs $H$ of $\fixGraph$ with at least one edge.
It is known from \cite[Theorem 2]{Ru88} that a central limit theorem for $W$ holds, as $n\to\infty$,
if and only if $\psi\to\infty$ and $n^2q\to\infty$.
This necessary and sufficient condition for asymptotic normality of $W$ is equivalent to the condition that $q\psi \to \infty$.
The corresponding Berry-Esseen bound is given in the following theorem.

\begin{theorem}\label{thm:subgraphs}
	Let $N\sim\mathcal{N}(0,1)$ be a standard Gaussian random variable.
	Then 
	\begin{align*}  
		d_K(W,N)  &= \cO( (q\psi)^{-\frac{1}{2}} ).
	\end{align*}
\end{theorem}

The rate of convergence in Theorem \ref{thm:subgraphs} has first been obtained for Wasserstein distance in \cite[Theorem 2]{BKR89} from a general bound for the normal approximation of so-called decomposable random variables. It has been a long standing problem whether the same rate could be achieved for Kolmogorov distance. In \cite{R03} the decomposition of \cite{BKR89} was considered and the author managed to prove a Berry-Esseen theorem, if all components of the decomposition are assumed to be bounded. But interesting enough, this does not lead to an optimal result as in Theorem \ref{thm:subgraphs}. Some special cases for Kolmogorov distance have been settled in \cite[Chapter 10]{modbook} (the case of fixed $p$), \cite[Theorem 1.1]{Rol17} and \cite[Theorems 1.1 and 1.2]{KRT17}.
In full generality this theorem has first been proven in \cite[Theorem 4.2]{PS20}
using the normal approximation bound from \cite[Proposition 4.1]{KRT17}
and a multiplication formula for discrete multiple stochastic integrals {(the additional assumption in \cite{PS20} that the graph $\Gamma$ has no isolated vertices is not necessary and can be removed as we explain at the beginning of the proof of Theorem \ref{thm:subgraphs} in Section \ref{sec:ApplSubgraphs})}.
Recently in \cite{ER01}, two of us derived an alternative proof, combining the decomposition of \cite{BKR89} with     elements of Stein's method
and with the theory of characteristic functions, sometimes called Stein-Tikhomirov method. In the present paper we will provide  yet another proof of this result, which is almost purely combinatorial and, as we think, conceptually easier than the one in \cite{PS20}. In fact, our proof is based on Theorem \ref{Kol_RAD},
which is a simplified version of the normal approximation bound in \cite[Proposition 4.1]{KRT17} used in \cite{PS20}. Additionally, our proof does not require a multiplication formula for discrete multiple stochastic integrals.

\medskip

We turn now to the number $V_d$ of vertices of $\mathbf{G}(n,p)$ having degree equal to $d$  for fixed $d\in\N_0:=\N\cup\{0\}$. To study the asymptotic normality of $V_d$, as $n\to\infty$, we define  
\begin{align}\label{def_F_d}
F_{d} := {V_{d}-\E V_{d}\over\sqrt{\var(V_{d})}}.
\end{align}
We know from \cite[Theorem 8]{BKR89} and \cite[Theorem 6.36]{JLRBook}   that, as $n\to\infty$,
\begin{itemize} 
\item a central limit theorem for $F_0$ holds if and only if $n^2p\to\infty$ and $np-\log n\to -\infty$;
\item  a central limit theorem for $F_d$ with $d\geq 1$ holds if and only if $n^{d+1}p^d\to\infty$ and $np-\log n-d\log\log n\to -\infty$. 
\end{itemize} The next result yields a corresponding Berry-Esseen bound, significantly improving the rate and the range of applicability of \cite[Theorem 1.3]{KRT17}.   We remark that the special case where $d\geq 1$ and $p$ is of the order $1/n$ has been treated in \cite[Theorem 2.1]{GoldsteinDegree} by much more sophisticated methods, while a Berry-Esseen bound for $d=0$ is the content of \cite{Kordecki}. 

\begin{theorem}\label{thm:VertexDegrees}
	Fix $d\in\N_0$ and let $N\sim\mathcal{N}(0,1)$ be a standard Gaussian random variable.
	\begin{itemize}
		\item[(a)] Assume that $np = \cO(1)$ and $n^2p\to\infty$, as $n\to\infty$. Then  
		\begin{align}\label{Kol_F_0}
			d_K(F_0,N) =   \cO(n^{-1} p^{-1/2}).
		\end{align}
		\item[(b)] Assume that $np = \cO(1)$ and $(np)^{-d}p^{1\over 2}\to 0$. Then
		\begin{align} \label{Kol_F_d}
			d_K(F_d,N) = \begin{cases}
				\cO\big(n^{-1/2}\big), & \text{if in addition $\liminf\limits_{n\to\infty}np>0$;} \\
				\cO\big((np)^{-d}p^{1\over 2}\big),  & \text{if in addition $np\to 0$}.
			\end{cases}
		\end{align}
	\end{itemize}
\end{theorem}

\begin{remark}
In particular, in the set-up studied in \cite{KRT17}, that is, if $p=\theta n^{-\alpha}$, where $\alpha \in [1,2)$ with $\theta=\theta(n) \in (0,n^\alpha)$ is such that $ \liminf\limits_{n \rightarrow \infty} \theta >0$, then Theorem \ref{thm:VertexDegrees} yields that
\[
d_K(F_0,Z) \leq \cO(n^{-1+\frac{1}{2}\alpha})
\]
for the number of isolated vertices. Moreover, if $d\geq 1$, we may choose $\alpha\in[1,1+{1\over 2d-1})$ and deduce that
$$
d_K(F_d,Z) = \cO(n^{-d-{\alpha\over 2}+\alpha d})\,.
$$
We would like to point out that this is not the full range of values of $\alpha$ for which the vertex counting statistics $F_d$ satisfy a central limit theorem. The range for $d\in\{0, 1\}$ is in fact optimal, while Theorem 6.36 in \cite{JLRBook} implies that for $d\geq 2$ a central limit theorem for $F_d$ applies if and only if $\alpha\in[1,1+1/d)$. Still, Theorem \ref{thm:VertexDegrees} improves \cite[Theorem 1.3]{KRT17}, which yields for $d\geq 1$ the considerably weaker bound of order $\cO(n^{1/2-3d/2-\alpha+3\alpha d/2})$ if $\alpha\in[1,1+1/(3d-2))$ (for $d=0$ the bound is the same as the one we get). Moreover, in the special case where $\alpha=1$, our bound in Theorem \ref{thm:VertexDegrees} shares the same quality as the corresponding results in \cite[Theorem 2]{Kordecki} and \cite[Theorem 2.1]{GoldsteinDegree}, but our argument is much simpler.
\end{remark}

\medspace

\begin{figure}[t]
\includegraphics[width=0.8\columnwidth]{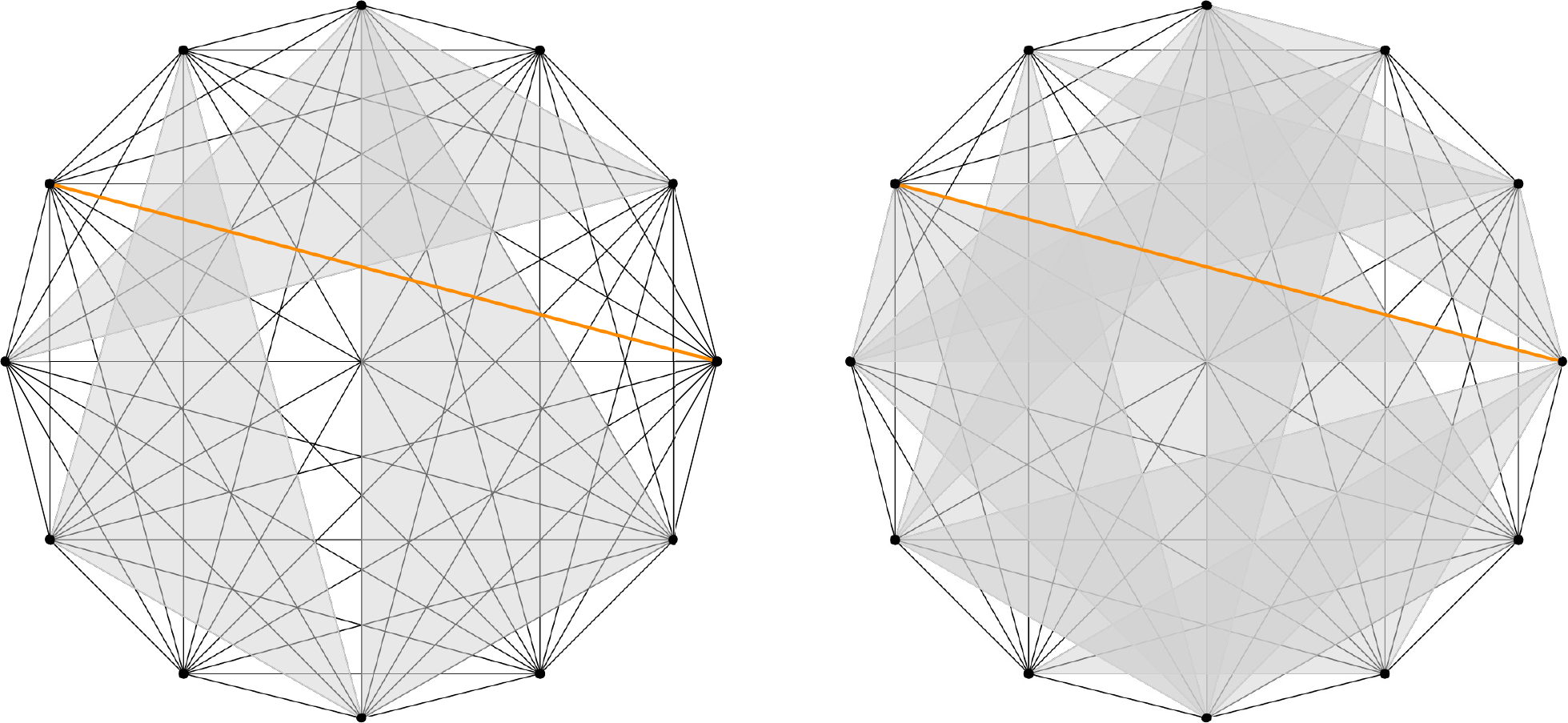}
\caption{Illustration of a random $2$-complex ${\bf Y}_2(n,p)$ with $n=12$ with small $p$ (left panel) and large $p$ (right panel). The orange edge is isolated on the left, while the same edge is not isolated any more on the right.}
\label{fig:2complex}
\end{figure}

\subsection{Application to the random $\kk$-complex}

Partially generalizing Theorem \ref{thm:VertexDegrees}, we are now going to discuss the number of isolated faces in the {Linial-Meshulam-Wallach random $\kk$-complex} $\mathbf{Y}_\kappa(n,p)$ for integers $\kappa\geq 1$ (see \cite{LinialMeshulam,MeshulamWallach}). A geometric realization of $\mathbf{Y}_\kk(n,p)$ is obtained by independently adding $\kk$-faces with probability $p\in[0,1]$ to the full $(\kk-1)$-skeleton of an $(n-1)$-dimensional simplex with $n$ vertices. Here,
a $\kappa$-face is a convex hull of $\kappa+1$ vertices of the $(n-1)$-dimensional simplex, for example, a $0$-face is just a vertex and a $1$-face is just an edge. We say that a $(\kk-1)$-face in the random complex $\mathbf{Y}_\kk(n,p)$ is isolated provided that it is not contained in the boundary of any $\kk$-simplex in ${\bf Y}_\kk(n,p)$ (in the literature such faces are also known as maximal faces). We remark that taking $\kk=1$, the random $\kk$-complex reduces to the Erd\H{o}s-R\'enyi random graph $\mathbf{G}(n,p)$ and the notion of an isolated $0$-face to that of an isolated vertex as discussed above. The isolated faces in ${\bf Y}_\kk(n,p)$ are of considerable interest in stochastic topology, because they are the last obstacle before homological connectivity is reached. In fact, it has been shown that the threshold $p=\kk(\log n)/n$ for the non-existence of isolated faces coincides with that of vanishing homology $H^{\kk-1} \big({\bf Y}_\kk(n,p);G \big)$ of order $\kk-1$ with values either in an arbitrary finite abelian group $G$ (see \cite{LinialMeshulam,MeshulamWallach}) or in $G=\Z$ (see \cite{LutzakPeled}). We remark that this parallels the behaviour for the Erd\H{o}s-R\'enyi random graph. The next theorem provides a Berry-Esseen bound for the number of isolated faces in ${\bf Y}_\kk(n,p)$ and generalizes part (a) of Theorem \ref{thm:VertexDegrees}. As above, we allow $p$ to depend on $n$ in what follows without highlighting this in our notation.

\begin{theorem}\label{thm:Isolated faces}
	Fix $\kk\in\N$. Let $I$ be the number of isolated  $(\kappa-1)$-faces in ${\bf Y}_\kk(n,p)$ and define $F:=(I-\E I)/\sqrt{\var(I)}$.  If  $np=\cO(1)$ and $n^{\kk+1}p\to\infty$, as $n\to\infty$, then
	\[
	d_K(F,N) =   \cO(n^{-{(\kk+1)}/2} p^{-1/2} ),
	\]
where
$N\sim\mathcal{N}(0,1)$ is a standard Gaussian random variable.
\end{theorem}

\medspace

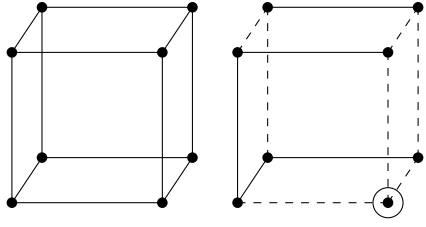
\begin{figure}[t]
	\begin{center}
		\begin{tikzpicture}
			\coordinate (A) at (0,0);
			\coordinate (B) at (2,0);
			\coordinate (C) at (2,2);
			\coordinate (D) at (0,2);
			\coordinate (E) at (0+0.4,0+0.6);
			\coordinate (F) at (2+0.4,0+0.6);
			\coordinate (G) at (2+0.4,2+0.6);
			\coordinate (H) at (0+0.4,2+0.6);
			\fill (A) circle (2pt);
			\fill (B) circle (2pt);
			\fill (C) circle (2pt);
			\fill (D) circle (2pt);
			\fill (E) circle (2pt);
			\fill (F) circle (2pt);
			\fill (G) circle (2pt);
			\fill (H) circle (2pt);
			\draw (A) -- (B);
			\draw (B) -- (C);
			\draw (C) -- (D);
			\draw (D) -- (A);
			\draw (E) -- (F);
			\draw (F) -- (G);
			\draw (G) -- (H);
			\draw (H) -- (E);
			\draw (A) -- (E);
			\draw (B) -- (F);
			\draw (C) -- (G);
			\draw (D) -- (H);
			
			\coordinate (A) at (3,0);
			\coordinate (B) at (2+3,0);
			\coordinate (C) at (2+3,2);
			\coordinate (D) at (0+3,2);
			\coordinate (E) at (0+3+0.4,0+0.6);
			\coordinate (F) at (2+3+0.4,0+0.6);
			\coordinate (G) at (2+3+0.4,2+0.6);
			\coordinate (H) at (0+3+0.4,2+0.6);
			\fill (A) circle (2pt);
			\fill (B) circle (2pt);
			\draw (B) circle (0.2cm);
			\fill (C) circle (2pt);
			\fill (D) circle (2pt);
			\fill (E) circle (2pt);
			\fill (F) circle (2pt);
			\fill (G) circle (2pt);
			\fill (H) circle (2pt);
			\draw [dashed] (A) -- (B);
			\draw [dashed] (B) -- (C);
			\draw (C) -- (D);
			\draw (D) -- (A);
			\draw (E) -- (F);
			\draw [dashed] (F) -- (G);
			\draw (G) -- (H);
			\draw [dashed] (H) -- (E);
			\draw (A) -- (E);
			\draw [dashed] (B) -- (F);
			\draw [dashed] (C) -- (G);
			\draw [dashed] (D) -- (H);
		\end{tikzpicture}
	\end{center}
	\caption{Hamming hypercube (for $n=3$) before and after the edge percolation, which has resulted in one isolated vertex, $4$ vertices of degree $1$, $3$ vertices of degree $2$ and no vertex of degree $3$.}
	\label{fig:hypercube}
\end{figure}

\subsection{Application to hypercube percolation} Let us now consider  another  random graph model for which the underlying graph is the $1$-skeleton of the $n$-dimensional hypercube. More precisely, its vertex set is
$$
\{-1,1\}^n=\{x=(x_1,\ldots,x_n)\in\R^n:|x_i|=1,i=1,\ldots,n\}
$$
and any two vertices $x=(x_1,\ldots,x_n)$,  $y=(y_1,\ldots,y_n)$ are connected, when  there is {exactly one} coordinate $i\in\{1,\ldots,n\}$ such that $x_i\neq y_i$, see Figure \ref{fig:hypercube}. On this Hamming hypercube,  we perform a percolation process in which each edge is removed independently with probability $1-p$ and kept with probability $p\in[0,1]$. The resulting random graph is denoted by $\mathbf{H}(n,p)$ and we emphasize that we allow $p$ to depend on $n$, although this will not be visible in our notation again.

Hypercube percolation has been introduced by Erd\H{o}s and Spencer \cite{ErdosSpencer} and subsequently been studied intensively. In particular, it has been observed that similarly to the Erd\H{o}s-R\'enyi graph $\mathbf{G}(n,p)$,  the random hypercube graph $\mathbf{H}(n,p)$ undergoes phase transitions. For example, for fixed $p$ the probability that $\mathbf{H}(n,p)$ is connected converges to $0$, $1/e$ or $1$, as $n\to\infty$, provided that $p<1/2$, $p=1/2$ or $p>1/2$, respectively. Moreover, from \cite{AtaijKomlosSzemeredi} it is known that if $p=\lambda/n$ with $\lambda>1$ there exists a unique `giant' component having size of order $2^n$, while for $\lambda<1$ the size of the largest component is of lower order. Here, we study the number of vertices of fixed degree $d\in\N_0$ in the random graph ${\bf H}(n,p)$ and we retain the notation $V_d$ for this number. In what follows we will assume that $p\to 0$ or $p\to1$ slower than exponential, meaning that for any $c\in(1,\infty)$
\begin{itemize}
	\item[(i)]$p\to 0$ and $\log p + c n \to\infty$ as $n\to\infty$ or 
	\item[(ii)]  $p\to 1$ and $\log (1-p) + c n \to\infty$ as $n\to\infty$.
\end{itemize}
In particular, this covers the the situation $p=\lambda/n$ at which the phase transition for $\mathbf{H}(n,p)$ takes places as discussed above.

\begin{theorem}\label{thm:IsolatedVerticesHypercube}
	Recall the definition of $F_d$ from  \eqref{def_F_d} and  let $N\sim\mathcal{N}(0,1)$ be a standard Gaussian random variable. Assume that     $p\to 0$ or $p\to 1$ slower than exponential, as $n\to\infty$. Then for fixed $d\in\N_0$,   it holds that  
	  for any $\e\in(0,1)$,
	\[
	d_K(F_d,N) =\cO\big( (2-\e)^{-{n\over 2}} \big)\,.
	\]
\end{theorem}

\medspace

Besides the application to random graphs, we would like to point out that our results are  relevant from a more theoretical point of view in the context of the celebrated {fourth moment theorem}. The latter says that a sequence of normalized multiple stochastic integrals of fixed order with respect to a Gaussian process or a Poisson process converges in distribution to a standard Gaussian random variable $N\sim\mathcal{N}(0,1)$, provided that their fourth moments converge to $3$, the fourth moment of $N$. However, for discrete multiple stochastic integrals of order two or higher,  the convergence of the fourth moment alone does not necessarily imply asymptotic normality as shown in \cite{DK19}. Instead, one also has to take into consideration the so-called maximal influence of the corresponding integrands. The new abstract bound for normal approximation we derive allows us to give a considerably simplified proof of the fourth-moment-influence bound in Kolmogorov distance \cite[Theorem 1.1]{DK19} as we explain in detail in Remark \ref{rem_FMT} below.

\medspace

The remaining parts of this paper are structured as follows. In Section \ref{sec:Prelim} we collect some preliminary material, especially including some important elements related to the discrete Malliavin formalism. The anticipated new normal approximation bound for general non-linear functionals of possibly infinite Rademacher random variables  is derived in Section \ref{sec:NormApproxBound}, where we also discuss the application to the fourth moment theorem. The simplified discrete Gaussian second-order Poincar\'e inequality is presented in Section \ref{sec:2ndorderpoincare}. The proof of Theorem \ref{thm:2runs} is presented in Section \ref{sec:2runs}, the proof of Theorem \ref{thm:subgraphs} is carried out in Section \ref{sec:ApplSubgraphs}, the proof of Theorem \ref{thm:VertexDegrees} is the content of Section \ref{sec:ApplERgraph}, Theorem \ref{thm:Isolated faces} is proved in Section \ref{sec:Complexes}, while Theorem \ref{thm:IsolatedVerticesHypercube} is shown in Section \ref{sec:Hypercube}.

\section{Preliminaries}\label{sec:Prelim}

\subsection{Notation}

For two sequences $a=a(n)$ and $b=b(n)$ we write $a\asymp b$ provided that
\[
c<\liminf_{n\to\infty}|a(n)/b(n)|\leq \limsup_{n\to\infty}|a(n)/b(n)|\leq C
\]
for two constants $0<c\leq C<\infty$.  If 
\[
\limsup_{n\to\infty} |a(n)/b(n)|<\infty,
\]
we write  $a=\cO(b)$. It is immediate that $a\asymp b$ if and only if $a= \cO(b)$ and $b= \cO(a)$.

We denote by $\N$ the set of positive integers and write $\N_0 =  \N\cup\{0\} $. The space of real square-summable sequences is a Hilbert space denoted by $\fH$. We say $a = (a(k), k\in\N)\in\fH$ if 
\[
\| a\|^2 := \sum_{k\in\N}   a(k) ^2 <\infty;
\]
the associated  inner product is given by 
\[
\langle a, b \rangle : = \sum_{k\in\N} a(k) b(k), ~ \text{for $a,b\in\fH$.}
\]
Sometimes, we may abuse the above notation and write 
\begin{align}\label{abu_nota}
\langle u, vw \rangle = \sum_{k\in\N} u(k) v(k) w(k),
\end{align}
whenever the above sum is well defined. Here $vw$ stands for the mapping $k\in\N\longmapsto v(k) w(k)$, which shall be distinguished from the tensor product $v\otimes w$. For $p\in\N$, we denote by $\fH^{\otimes p}$ ($\fH^{\odot p}$ respectively) the $p$th tensor product (symmetric tensor product {resp.}) of $\fH$.

\subsection{Discrete Malliavin formalism}

 Let $(p_k, k\in\N)$ be a sequence of real numbers in $(0,1)$ and define  $q_k =1 - p_k$, $k\in\N$.   Suppose $\mathbf{X}=(X_k, k\in\N)$ is a sequence of independent Rademacher random variables with 
 \begin{align*}
\PP(X_k = 1)=p_k\qquad\text{and}\qquad\PP(X_k = -1)=q_k,\qquad k\in\N.
 \end{align*}
Let  
\begin{align}\label{Yk}
\left(Y_k := \frac{ X_k - p_k +q_k}{ 2\sqrt{p_kq_k}} ,\, k\in\N   \right)
\end{align}
 be the normalized sequence.  
 
 \subsubsection{Chaos expansion} The Wiener-It\^o-Walsh  decomposition theorem asserts that the space  $L^2(\Omega, \sigma\{\mathbf{X}\}, \PP )$ with $\sigma\{\mathbf{X}\}$ denoting the $\sigma$-field generated by the random sequences $\bf X$ can be expressed as a direct sum of mutually orthogonal subspaces  (see {e.g.} \cite[Proposition 6.7]{Privault08}):
\begin{align}\label{chaos_d}
L^2(\Omega, \sigma\{\mathbf{X}\}, \PP ) = \bigoplus_{p\geq 0} \C_p,
\end{align}
where $ \C_0 =\R$ and for $p\in\N$, $\C_p$ is called   the $p$th {Rademacher chaos}, which is the collection of square-integrable  $p$-linear polynomials in $\{Y_k, k\in\N\}$. More precisely, 
\[
\C_p = \big\{ J_p(f) : f\in \fH^{\otimes p} \big\},
\]
where for $f\in \fH^{\otimes p}$,  $ J_p(f) $ is called the $p$th {discrete multiple integral} of $f$ and is defined by
\[
 J_p(f) : = \sum_{(i_1, ... , i_p)\in\N^p} f(i_1,  ... , i_p) Y_{i_1} Y_{i_2} \cdots Y_{i_p} \1_{\Delta_p}(i_1, ..., i_p),
 \]
with   $\Delta_p := \{  (i_1, ..., i_p) \in\N^p: i_j \neq i_k~\text{for $j\neq k$} \}$. By definition, we have $J_p(f)  = J_p\big(\widetilde{f} \big) $, where $\widetilde{f}\in\fH^{\odot p} $ stands for the canonical symmetrization
\[
\widetilde{f}(i_1, ... , i_p ) := \frac{1}{p!} \sum_{\sigma\in\mathfrak{S}_p} f(i_{\sigma(1)}, ..., i_{\sigma(p)}),
\]
of $f$, with $\mathfrak{S}_p$ denoting the set of permutations over $\{1, ... , p\}$. 

The orthogonality   of different Rademacher chaoses is captured by the following (modified) isometry relation: 
\begin{align}\label{miso}
 \E\big[ J_p(f) J_q(g) \big] = \1_{\{ p =q\}} p! \langle \widetilde{f}, \widetilde{g} \1_{\Delta_q} \rangle_{\fH^{\otimes p}}~\text{for any $f\in\fH^{\otimes p}$ and any  $g\in\fH^{\otimes q}$}. 
\end{align}
With the notation $\fH^{\odot p}_0 = \big\{ f\in \fH^{\otimes p} :  f = \widetilde{f} \1_{\Delta_p} \big\}$ for $p\in\N$, we can rephrase \eqref{chaos_d} as follows. For any $F\in L^2\big(\Omega, \sigma\{\mathbf{X}\}, \PP \big)$, there is a {unique} sequence of kernels $f_n\in \fH^{\odot n}_0, n\in\N$  such that
\begin{align}\label{exp_F}
F= \E[F] +  \sum_{n =1}^\infty J_n(f_n) \qquad\text{in $L^2(\Omega)$. }
\end{align}

\bigskip

In the sequel, we  introduce some basic discrete Malliavin calculus and refer readers to the survey \cite{Privault08} for further details and background material.

\subsubsection{Discrete Malliavin operators}      For $F\in L^1(\Omega, \sigma\{\mathbf{X}\}, \PP)$, we can write $F = \mathfrak{f}(X_1, X_2,...  )$. We   define 
\[
F^+_k :=  \mathfrak{f}(X_1, ..., X_{k-1}, +1, X_{k+1}, ...  ) \quad{\rm and}\quad F^-_k :=  \mathfrak{f}(X_1, ..., X_{k-1}, -1, X_{k+1}, ...  ).
\]
The {discrete gradient} $D_kF$ of $F$ at $k$th coordinate is a  real-valued    random variable given by 
\[
D_kF := \sqrt{p_kq_k}  \big(F^+_k  - F^-_k \big), \, ~ k\in\N.
\]      
For example, $D_k Y_j = \1_{\{ k=j \}}$. The discrete gradient satisfies the following product formula. For $F, G\in L^2(\Omega, \sigma\{\mathbf{X}\}, \PP)$,
\begin{align}\label{prod}
D_k(FG) = G D_kF + F D_k G - \frac{X_k}{\sqrt{p_kq_k} } (D_kF) (D_k G), \qquad k\in\N;
\end{align}
see {e.g.} \cite[Proposition 7.8]{Privault08} for a proof.

\begin{lemma}\label{lem_0}
{\rm (1)} If $F\in L^\ell(\Omega, \sigma\{\mathbf{X} \}, \PP)$ for some $\ell\in [1, \infty)$, then $D_k F\in L^\ell(\Omega)$ for every $k\in\N$. 

\medskip

{\rm (2)} If $G\in L^\infty(\Omega)$, then
\begin{align}\label{L_inf_bdd}
\| D_kG\| _{L^\infty(\Omega)} \leq \| G\|_{L^\infty(\Omega)}.
\end{align}
If we assume additionally  {\rm (a)} $\sum_{k\in\N} p_k q_k <\infty$ or {\rm (b)} $G$ depends only on  $\{X_1, ... , X_m\}$, then  
\begin{align*}
 \E\big[ \| DG \|^2 \big] = \sum_{k\in\N} \E\big[ (D_kG)^2 \big]   \leq  \begin{cases}
{\displaystyle 4  \| G\|_{L^\infty(\Omega)}^2 \sum_{k\in\N} p_k q_k} &\text{in case {\rm (a) }} \\
& \quad\\
{\displaystyle 4  \| G\|_{L^\infty(\Omega)}^2 \sum_{k=1}^m p_k q_k} &\text{in case {\rm (b)}.} 
 \end{cases}
\end{align*}

\end{lemma}
\begin{proof} For $\ell\in[1,\infty)$,
we can write 
\begin{align*}
	\| F\|_{L^\ell(\Omega)}^\ell = \E\big[  |F|^\ell  \1_{\{ X_k=1\} } + |F|^\ell \1_{\{ X_k=-1\} }\big] &=  \E\big[  \big| F^+_k \big|^\ell  \1_{\{ X_k=1\} } +   \big| F^-_k \big|^\ell\1_{\{ X_k=-1\} }\big] \\
	&= p_k  \E\big[  | F^+_k|^\ell \big] + q_k  \E\big[  | F^-_k |^\ell\big], 
\end{align*}
where the last step follows from the independence between $F^\pm_k$ and $X_k$. Thus, the finiteness of $\| F\|_{L^\ell(\Omega)}^\ell $ implies  $F^+_k, F^-_k\in L^\ell(\Omega)$. 
Therefore, $D_k F = \sqrt{p_k q_k}\big( F^+_k  -  F^-_k \big)\in L^\ell(\Omega)$.

 The  $L^\infty$-case can be proved in the same manner. Suppose that  $G$ is essentially bounded by $M$ for some   $M\in(0,\infty)$, then both $G^+_k$ and $G^-_k$ are bounded by $M$ so that
\begin{align}
| D_kG | =  \sqrt{p_kq_k} \big\vert G^+_k  -  G^-_k \big\vert & \leq 2M   \sqrt{p_kq_k} \label{use_1} \\
&\leq M, \quad\text{using $p_k q_k = p_k (1-p_k) \leq 1/4$}. \notag
 \end{align}
 This shows \eqref{L_inf_bdd}. Under the additional assumption (a), we deduce from \eqref{use_1} that
\[
 \E\big[ \| DG \|^2 \big] = \sum_{k\in\N} \E\big[ (D_kG)^2 \big]  \leq 4M^2 \sum_{k\in \N}p_k q_k.
\]
The other case is  also straightforward.   \qedhere

\end{proof}

Let $\mathbb{D}^{1,2}$ denote  the set of real random variables $F\in  L^2\big(\Omega, \sigma\{\mathbf{X}\}, \PP\big)$ with
\[
\E\big[ \| DF \|^2 \big]  = \E \left( \sum_{k\in\N} (D_kF)^2 \right) < \infty.
\]
The space $\mathbb{D}^{1,2}$ is a Hilbert space under the norm $\| F\|_{1,2} := \sqrt{ \E[ F^2] + \E [ \|DF \|^2] }$ and the set \[ \mathcal{S} := \bigcup_{n\geq 1} \bigoplus_{p=0}^n \C_p\]is a dense subset of $\mathbb{D}^{1,2}$.

The adjoint  operator $\delta$ of $D$ is characterized by the duality relation
\begin{align}\label{duality}
\E\big[ \langle DF, u \rangle \big] = \E\big[  F \delta( u)  \big],
\end{align}
and  its domain $\text{Dom}(\delta)$   consists of   square-integrable $\fH$-valued random variables   $u\in L^2(\Omega;\fH )$ satisfying the following property:
 \begin{center}
 there is some constant $C_u > 0$ such that $\big\vert   \E[ \langle DF, u \rangle  ]  \big\vert \leq C_u  \sqrt{ \E[ F^2]  }$, for all $F\in\mathbb{D}^{1,2}$.
\end{center}
For $F\in L^2(\Omega, \sigma\{ \mathbf{X} \}, \PP)$ having the representation \eqref{exp_F}, it is not difficult to see that
\[
D_k F =  \sum_{n\geq 1} n J_{n-1}(f_n(k,\,\cdot\,) ), 
\]
where $J_0(f_1(k,\,\cdot\,)) = f_1(k)$ and for $n\geq 2$, $J_{n-1}(f_n(k,\,\cdot\,) )$ is the $(n-1)$th discrete multiple integral of $f_n(k, \,\cdot\,)\in\fH^{\odot (n-1)}_0$. Then using the modified isometry property \eqref{miso}, we deduce by comparing $\var(F)$ and $\E[ \| DF \|_\fH^2]$ that 
\begin{align}\label{GPI}
\var(F) \leq  \E[ \| DF \|_\fH^2]
\end{align}
for all $F\in\mathbb{D}^{1,2}$. The above inequality is known as the {Gaussian Poincar\'e inequality} in the Rademacher setting and we note that  the inequality \eqref{GPI} reduces to an equality if and only if $F\in \C_0\oplus \C_1$.

\medskip

Suppose $u = (u_k, k\in\N)\in\text{Dom}(\delta) \subset L^2(\Omega ; \fH)$, then for each $k\in\N$, $u_k\in L^2(\Omega)$ admits the chaos expansion
\begin{align}
u_k = \E[ u_k] + \sum_{n=1}^\infty J_{n}\big( g_{n+1}(k, \,\cdot\,) \big), \label{uk_exp}
\end{align}
where $g_{n+1}(k, \,\cdot\,)\in\fH^{\odot n}$ and $g_{n+1}\in\fH^{\otimes (n+1)}$. If  $F\in\mathcal{S}$ has the form $F = \E[ F] + \sum_{n=1}^m J_n(f_n)$ with $f_n\in\fH^{\odot n}_0$, then
\[
D_kF =  \sum_{n=1}^m n J_{n-1}\big(f_n(k,\,\cdot\,) \big)
\] 
so that 
we deduce from
   \eqref{duality} and \eqref{miso}  that
   \begin{align*}
 \E\big[ \langle DF, u \rangle\big]  =\sum_{k\in\N} \E[  u_k D_kF ] &= \sum_{k\in\N} \sum_{p=1}^m p! \langle  f_p(k, \,\cdot\,), g_p(k, \,\cdot\,) \rangle_{\fH^{\otimes (p-1)}} \\
 & = \sum_{p=1}^m p! \langle f_p, g_p\rangle_{\fH^{\otimes p}} = \E\left[ F \sum_{p=1}^m J_p\big( \widetilde{g}_p\1_{\Delta_p} \big) \right].
   \end{align*} 
It follows that   $u\in\text{Dom}(\delta)$ if and only if $\sum_{p=1}^\infty  p! \big\|  \widetilde{g}_p\1_{\Delta_p}  \big\|^2 <\infty$. In this case,
\[
\delta(u)=  \sum_{p=1}^\infty J_p\big( \widetilde{g}_p\1_{\Delta_p} \big).
\]

Next, we define the Ornstein-Uhlenbeck operator $L$ and its pseudo-inverse $L^{-1}$.  Suppose $F\in L^2(\Omega, \sigma\{ \mathbf{X} \}, \PP)$ has the representation \eqref{exp_F}, we say $F\in\text{Dom}(L)$ if  $\sum_{n\in\N} n^2n! \| f_n \|^2 < \infty$. In this case, we define
\[
LF := \sum_{n=1}^\infty -n J_n(f_n)
\]
and the associated semigroup $(P_t, t\in\R_+)$ is given as
\[
P_tF := \sum_{n=0}^\infty e^{-nt} J_n(f_n).
\]
 Also, for  $F\in L^2(\Omega, \sigma\{ \mathbf{X} \}, \PP)$ having the representation \eqref{exp_F}, we put
 \[
 L^{-1} F :=  \sum_{n=1}^\infty -\frac{1}{n} J_n(f_n)
 \]
so that $LL^{-1} F = F - \E[ F]$ for any $F\in L^2(\Omega, \sigma\{ \mathbf{X} \}, \PP)$. It is not difficult to verify that $F\in\text{Dom}(L)$ if and only if $F\in\mathbb{D}^{1,2}$ and $DF\in\text{Dom}(\delta)$; in this case, we can write $L = - \delta D$.

Finally, let us record a useful result from \cite{KRT17}.

\begin{lemma}\label{lem_2} For $k\in\N$ and $F\in L^2(\Omega, \sigma\{\mathbf{X}\}, \PP)$, we have 
\begin{align}\label{int_form}
-D_k L^{-1} F = \int_0^\infty e^{-t} P_t ( D_k F) dt.
\end{align}
Moreover, if $D_k F\in L^q(\Omega)$ for some $q\in[2,\infty)$, then $\| D_k L^{-1} F \| _{L^q(\Omega)} \leq \| D_k F \| _{L^q(\Omega)}$.
\end{lemma}
 
 \begin{proof} The integral representation of $-D_k L^{-1}F$ is an easy consequence of  the chaos expansion. Assume now  $D_k F\in L^q(\Omega)$ for some $q\in[2,\infty)$, then by Minkowski's inequality and the integral representation \eqref{int_form}, we have 
 \[
 \big\| D_k L^{-1} F \big\|_{L^q(\Omega)} \leq  \int_0^\infty e^{-t} \big\| P_t ( D_k F) \big\|_{L^q(\Omega)} dt \leq \| D_k F \|_{L^q(\Omega)},
 \]
 where the last step follows from the contraction property of $P_t$. We refer to \cite[Proposition 3.1-3.3]{KRT17} for more details.   \qedhere
 \end{proof}

\section{Normal approximation bounds for Rademacher functionals}\label{sec:NormApproxBound}

Using a discrete version of the Malliavin-Stein technique, a first normal approximation bound for symmetric Rademacher functionals has been obtained in the paper \cite{NourdinPeccatiReinert}, where the error bound was described 
 in terms of a probability metric based on smooth test functions.  A corresponding bound in  the Kolmogorov distance has later been found in \cite{KRT16} and was extended to the non-symmetric setting in \cite{KRT17} and in \cite{DK19}.  We will significantly simplify these bounds applying a monotonicity property of the solutions of Stein's equation in normal approximation. This approach was introduced in \cite{ShaoZhang19} for normal and non-normal approximations for unbounded exchangeable pairs and in \cite{Shao19} for further approaches of Stein's method. For functionals of Poisson random measures, the Malliavin-Stein method has recently been used in \cite{LRPY21} to deduce a simplified Berry-Esseen bound using the same monotonicity argument. We adapt these observations to the Rademacher set-up, which leads to the following result, which is essentially a simplified version of \cite[Theorem 3.1]{KRT16} (in the symmetric case), \cite[Proposition 4.1]{KRT17} and \cite[Proposition 4.2 and Theorem 1.1]{DK19} (in the non-symmetric case). We also point out that by using a suitable chain rule, a normal approximation bound in Wasserstein distance has been obtained in \cite{Zheng17}, whose order is comparable to the Kolmogorov bounds in \cite{KRT16, KRT17}.

Before we can proceed, we need to introduce some further notation. For fixed $z\in\R$, let $f_z$ be the solution to the Stein equation
\begin{equation} \label{SE}
	f'(x) - xf(x) = \1_{\{x\leq z\}} - \Phi(z),\qquad x\in\R,
\end{equation}
with $\Phi(\,\cdot\,)$ denoting the distribution function of a standard Gaussian random variable. The function $f_z$ satisfies the following properties (see Lemma 2.2 and Lemma 2.3 in \cite{CGS11}): 
\begin{itemize}
	\item[(i)]  $| f_z(x)| \leq  \sqrt{2\pi}/4$ for all $x\in\R$;
	\item[(ii)] $f_z(x)$ is continuous on $\R$, infinitely differentiable on $\R\setminus\{z\}$, but not differentiable at $x=z$;
	\item[(iii)] interpreting the derivative of $f_z(x)$  at $x=z$ as $f_z'(z) = z f_z(z) + 1 - \Phi(z)$, one has that $| f_z'(x) | \leq 1$ for all $x\in\R$;
	\item[(iv)] $|x f_z(x) | \leq 1$ for all $x$ and the map $x\longmapsto xf_z(x)$ is non-decreasing.
\end{itemize}	
Following the standard route in Stein's method we have the following bound in Kolmogorov distance. If $F$ has mean zero and variance one, then, {with $N\sim\mathcal{N}(0,1)$,}
\begin{align}\label{def_Kol}
	d_K(F, N) = \sup_{z\in\R} | \PP(F\leq z) - \Phi(z) | \leq \sup_{z\in\R} \Big\vert  \E\big[ F f_z(F) - f_z'(F) \big] \Big\vert.
\end{align}
To be able to deduce bounds simplifying those in \cite{KRT17}, property (iv) as well as the property that the mapping $x\longmapsto {\rm\mathbf{1}}_{\{ x > z \}} $ is non-decreasing will be the basis. Property (iv) was considered in \cite[Lemma 2.2]{CS05} the first time, while already in \cite{CS01} it has been used to prove a non-uniform Berry-Ess\'een bound for sums of independent and not necessarily identically distributed random variables.

 \begin{theorem} \label{Kol_RAD}
{\rm (1)} Let  $F\in\mathbb{D}^{1,2}$ have mean zero and variance one  such that  
	\begin{align}\label{condition_1}
	Ff_z(F) + {\rm\mathbf{1}}_{\{ F>z\}} \in\mathbb{D}^{1,2}\;\;\text{for all } z\in\R.
	\end{align}
	Then with  $N\sim \mathcal{N}(0,1)$, one has that
	 \begin{align}\label{Kol_R0}
		d_K(F, N) \leq    \E \Big[ \big\vert   1  - \langle DF,  - DL^{-1}F\rangle \big\vert \Big] +  \sup_{z \in \R}
		\E \Big[ \big\langle D \big( Ff_z(F) + {\rm\mathbf{1}}_{\{ F>z\}} \big)   \frac{D F}{\sqrt{p q}} ,  | D L^{-1} F | \big\rangle \Big].
		\end{align}
{\rm (2)} Assume in addition that $u\in\text{\rm Dom}(\delta)$, where
	\begin{align}\label{def:u_k}
	u_k:=(p_kq_k)^{-1/2} D_kF| D_kL^{-1}F|, \, k\in\N.
	\end{align}
	Then  for   $N\sim \mathcal{N}(0,1)$ one has that
	\begin{align}
		d_K(F, N) \leq    \E\Big[ \big\vert   1  - \langle DF,  - DL^{-1}F\rangle \big\vert \Big] +  2 \Big\| \delta\big((pq)^{-1/2} DF| DL^{-1}F| \big)  \Big\|_{L^1(\Omega)}.   \label{Kol_R1}
	\end{align}
{\rm (3)} Let $F\in\mathbb{D}^{1,2}$ have mean zero and variance one  such that  
	\begin{align}\label{A1}
		\text{$D_kF\in L^4(\Omega)$ for each $k\in\N$.}  
	\end{align}
	Assume either {\rm (a)} $\sum_{k\in\N} p_k q_k <\infty$ or {\rm (b)} $F$ only depends  on the first $m$ Rademacher random variables.
	Then   for   $N\sim \mathcal{N}(0,1)$ one has that
\begin{align}\label{Kol_R2}
		d_K(F, N) \leq    \E\Big[ \big\vert   1  - \langle DF,  - DL^{-1}F\rangle \big\vert \Big]+  4 \sqrt{\kappa} \left( \sum_{k\in\N} \frac{1}{p_kq_k} \E\big[ (D_k F)^4 \big]   \right)^{1/2},    
	\end{align}	
	where $\kappa = \sum_{k\in\N} p_k q_k$ in case {\rm (a)} and  $\kappa = \sum_{k=1}^m p_k q_k$ in case {\rm (b)}.
\end{theorem}

\begin{remark}\label{rem1} 
\begin{itemize}
\item[(1)] The estimates \eqref{Kol_R0} and \eqref{Kol_R1} simplify \cite[Proposition 4.1]{KRT17}, see \cite[Equation (4.9)]{KRT17}.

\item[(2)]  In view of   property  (iv) of the Stein solution $f_z$ {mentioned above} and Lemma \ref{lem_0},  
the condition \eqref{condition_1} is always satisfied if $\sum_{k\in\N} p_k q_k <\infty$ or $F$ depends only on finitely many $X_k$'s. 

\item[(3)] For the random sequence $u$ defined as in \eqref{def:u_k}, a sufficient condition for $u\in\text{Dom}(\delta)$ is that  $F\in L^\infty(\Omega)$ depends only on finitely many $X_k$'s.

\item[(4)] A sufficient condition for \eqref{A1} is that $F\in L^4(\Omega, \sigma\{\mathbf{X} \}, \PP)$; see Lemma \ref{lem_0}.  If  $p_k\to 0$, as $k\to\infty$,  then $\sum_{k\in\N} p_k q_k <\infty$ if and only if $\sum_{k\in\N} p_k <\infty$.
\end{itemize}	
\end{remark}

\begin{proof}[Proof of Theorem \ref{Kol_RAD}] 

	Using the identity $F= LL^{-1}F = -\delta DL^{-1}F$ and  the duality relation \eqref{duality}, we write
	\[
	\E\big[ Ff_z(F) \big] = -  \E\big[ \big( \delta DL^{-1} F \big) f_z(F) \big]  = \E\big[ \langle Df_z(F), -DL^{-1}F \rangle \big],
	\]
	so that
	\begin{align}\label{return1}
		\E\big[ f_z'(F) - F f_z(F) \big] = \E\big[ f_z'(F) - \langle Df_z(F), -DL^{-1}F \rangle \big].
	\end{align}
	By definition of the discrete gradient and the fundamental theorem of calculus, for $k\in\N$ we have 
	\begin{align}
		D_k f_z(F) &=\sqrt{p_kq_k}\Big[  f_z(F^+_k) -  f_z(F^-_k) \Big]  = \sqrt{p_kq_k} \int_0^{F^+_k - F^-_k} f_z'(t + F^-_k) \,dt \notag \\
		&=  \sqrt{p_kq_k} \int_0^{F^+_k - F^-_k} \big[  f_z'(t + F^-_k) - f_z'(F)  \big]  \,dt +  f'_z(F)D_kF := R_k +f'_z(F)D_kF ,\label{def_Rk}
	\end{align}
	where we follow the convention that, for a function $\phi:\R\to\R$,
	\begin{align}\label{rule0}
		\int_0^x \phi(t) \,dt = - \int_0^{-x} \phi(-t)\,dt \quad {\rm for}~ x<0.
	\end{align}
	Then, using the Stein's equation $f'_z(x) = x f_z(x) + \1_{\{ x\leq z \}} - \Phi(z)$, we get 
	\begin{align}\label{TRICK}
	f_z'(t + F^-_k) - f_z'(F) = (t+ F^-_k ) f_z( t+ F^-_k ) - Ff_z(F) +  \1_{\{ t+ F^-_k \leq z \}} - \1_{\{ F\leq z\}}.
	\end{align}
	Note that when $F^+_k = F^-_k$ or equivalently $D_kF = 0$,  we have  $D_k f_z(F)=R_k = 0$. Now, we distinguish the cases $F^+_k > F^-_k$ (Case 1) and $F^+_k < F^-_k$ (Case 2).
	
	\medspace
	
	\paragraph{\textbf{Case 1}} If $F^+_k > F^-_k$ or equivalently $D_kF > 0$, we have $F^+_k \geq F\geq F^-_k$. 
	Then using the monotonicity of the function $x\longmapsto xf_z(x)$ on $\R$, recall property (iv) of the Stein solution $f_z$, we get for $0\leq t \leq F^+_k - F^-_k$ that
	\begin{align*}
	\begin{cases}
			(t+ F^-_k ) f_z( t+ F^-_k ) - Ff_z(F) &\leq\quad   F^+_k  f_z(  F^+_k ) - Ff_z(F) \in \R_+ \\
			\quad\\
			Ff_z(F) - ( t+ F^-_k ) f_z( t+ F^-_k ) &\leq\quad    Ff_z(F) -  F^-_k  f_z(F^-_k )   \in \R_+.
		\end{cases}
	\end{align*}  
	Combining these two inequalities yields
	\[
	| (t+ F^-_k ) f_z( t+ F^-_k ) - Ff_z(F) | \leq  \frac{1}{\sqrt{p_kq_k}}  D_k \big( Ff_z(F) \big).
	\]
	Note that the function $x\longmapsto \1_{\{ x> z \}}$ is also non-decreasing on $\R$, so that the same arguments lead to 
	\[
	\big\vert   \1_{\{ t+ F^-_k \leq z \}} - \1_{\{ F\leq z\}} \big\vert  = \big\vert   \1_{\{ t+ F^-_k > z \}} - \1_{\{ F> z\}} \big\vert \leq \frac{  D_k\big(  \1_{\{ F> z\}}  \big)}{   \sqrt{p_kq_k}  }
	\]
	for $0\leq t \leq F^+_k - F^-_k$.
	Therefore,  recalling the definition of $R_k$ in \eqref{def_Rk}, we have that
	\begin{align}\label{com1}
		\left\vert  R_k   \right\vert  \leq   \frac{1}{\sqrt{p_kq_k}}  D_k \big( Ff_z(F) + \1_{\{ F>z\}} \big)D_kF.
	\end{align}

	\paragraph{\textbf{Case 2}} If $F^+_k < F^-_k$ or equivalently $D_kF <  0$, we have $F^+_k \leq F\leq  F^-_k$.  Taking into account our  convention \eqref{rule0}, we write
	\begin{align*}
		R_k &= \sqrt{p_kq_k} \int_0^{ F^-_k-F^+_k}  \big[ f_z'(F) -  f_z'(F^-_k -t) \big] \,dt,
	\end{align*}
	and by the same arguments as in \textbf{Case 1}, we have
	\[
	\big\vert  f_z'(F) -  f_z'(F^-_k -t) \big\vert \leq -\frac{1}{\sqrt{p_kq_k}} D_k\big( Ff_z(F) + \1_{\{ F>z\}} \big)
	\]
	so that the estimate \eqref{com1} holds in this case as well.
	
	\medspace 
	
	Thus, by combining the above discussions with  \eqref{return1}, we get 
	\begin{align*}
		\E\big[ f_z'(F) - F f_z(F) \big] &= \E\big[ f_z'(F) - f_z'(F) \langle DF, -DL^{-1}F  \rangle \big] +  \E\big[   \langle R, -DL^{-1}F  \rangle \big].
	\end{align*}
	which is bounded by 
	\begin{align*}
 \E\Big[  \big\vert 1 -  \langle DF, -DL^{-1}F  \rangle \big\vert \Big] +  \sum_{k\in\N} \E\Big[  | R_k| \times  | D_kL^{-1} F | \Big].
	\end{align*}
  Also, by \eqref{com1}, we get
	\begin{align}\label{need1}
		\sum_{k\in\N} \E\Big[  | R_k| \times  | D_kL^{-1} F | \Big] \leq  \sum_{k\in\N} \E\Big[ D_k \big( Ff_z(F) + \1_{\{ F>z\}} \big)   \frac{D_kF}{\sqrt{p_kq_k}}        | D_kL^{-1} F | \Big].
	\end{align}
	Note that the above sum is finite, since $D_k \big( Ff_z(F) + \1_{\{ F>z\}} \big)/ \sqrt{p_kq_k} $ is bounded by $2$ and 
	\[
	\sum_{k\in\N} \E\Big[  | D_kF |     | D_kL^{-1} F | \Big] \leq \sum_{k\in\N} \E\big[  | D_kF |^2 \big]   <\infty 
	\]
	by  Lemma \ref{lem_2}.  Part (1) is thus proved.
	
	\medskip
	
	To prove (2), putting    
	\begin{align}\label{def_G_u_k}
		G=  Ff_z(F) + \1_{\{ F>z\}} \quad{\rm and}\quad u_k =   \dfrac{D_kF}{\sqrt{p_kq_k}}    | D_kL^{-1} F | \,\, \text{for $k\in\N$,}
\end{align}
 we have the following facts:
	\begin{itemize}
		\item[(i)]  $G$ is uniformly bounded by $2$ and  $ (D_k G) u_k \geq 0 $;
		\item[(ii)]
		$u_k\in L^2(\Omega)$ for all $k\in\N$, since ${\displaystyle
\E\big[ \| u\|_{\mathfrak{H}}^2\big] =	\E \sum_{k\in\N} u_k^2 <\infty.
	}$
	\end{itemize}
 
By assumptions in part (1) and part (2),   $Ff_z(F) + \1_{\{ F>z\}}\in\mathbb{D}^{1,2}$ and    $u\in\text{Dom}(\delta)$, then we deduce from     the duality relation  \eqref{duality} that 
	\begin{align*}
	\sum_{k\in\N} \E\Big[ (D_k G) u_k \Big] &= {\E\Big[ \langle DG, u \rangle_{\mathfrak{H}} \Big]  }  \quad \text{using fact (i)}\\
	&= \E\Big[  G \delta\big(  (pq)^{-1/2} DF | DL^{-1}F| \big) \Big] \leq 2 \big\| \delta\big(  (pq)^{-1/2} DF | DL^{-1}F| \big) \big\|_{L^1(\Omega)}.
	\end{align*}
	Thus,  we obtain  \eqref{Kol_R1} from the Stein bound \eqref{def_Kol}.

	\medskip
	
	Now let us consider part (3).  Following the estimate in \eqref{need1} and using the Cauchy-Schwarz inequality, we  get, with $G = Ff_z(F) + \1_{\{ F>z\}}$ as above,
\begin{align*}
		\sum_{k\in\N} \E\Big[  | R_k| \times  | D_kL^{-1} F | \Big] &\leq  \left( \sum_{k\in\N} \E\big[ (D_k G)^2 \big]  \right)^{1/2}  \left( \sum_{k\in\N}  \frac{1}{p_kq_k} \E\Big[  (D_kF)^2       | D_kL^{-1} F |^2 \Big] \right)^{1/2} \\
		&\leq  \left( 16 \kappa \right)^{1/2}\left( \sum_{k\in\N}  \frac{1}{p_kq_k} \E\big[  (D_kF)^4      \big] \right)^{1/2},
		\end{align*}	
		where the last step follows from Lemma \ref{lem_0}, Lemma \ref{lem_2} and the fact that $G$ is uniformly bounded by $2$. Then the   bound \eqref{Kol_R2} follows from  the Stein bound \eqref{def_Kol}. 
		This concludes the proof. \qedhere

\end{proof}

\begin{remark}
The second derivatives of the solutions of Stein's equations \eqref{SE} are unbounded. Therefore, to obtain good bounds for the Kolmogorov distance,
increments like $f_z'(u) - f_z'(v)$ should not be represented in terms of second derivatives by the mean value theorem. The idea
in \eqref{TRICK} is alternatively to rewrite these increments using Stein's equation. As a consequence, one has to deal with terms like $ (w+u) f_z(w+u) - (w+v) f_z(w+v)$.
Applying the property that $x\mapsto xf_z(x)$ is non-decreasing one can show the bound
$$
 | (w+u) f_z(w+u) - (w+v) f_z(w+v) | \leq \Big( |w| + \frac{\sqrt{2 \pi}}{4} \Big) (|u| + |v|).
$$
Alternatively, a direct application of the property that $x \mapsto xf_z(x)$ is non-decreasing is the key to the simplification of our bounds.
\end{remark}

In \cite[Proposition 4.2]{DK19} bounds for normal approximation of non-linear functionals of an infinite Rademacher sequence in the Kolmogorov metric  are presented in terms of an operator $\Gamma_0$, which coincides with the so-called carr\'e-du-champ operator $\Gamma$ for functionals in $L^4(\Omega)$ having a finite chaotic decomposition (see \cite[Proposition 2.7]{DK19}). We will now establish a simplified version of this result. For this, we define for $F,G  \in\mathbb{D}^{1,2}$,
\begin{align}\label{Gamma0}
	\Gamma_0(F,G) := \frac 12 \sum_{k=1}^{\infty} (D_k F)(D_k G) +  \frac 12 \sum_{k=1}^{\infty} (D_k F)(D_k G) Y_k^2
\end{align}
with $Y_k$ as in \eqref{Yk}.

\begin{theorem} \label{Kol_RAD2}
	Let  $F\in\mathbb{D}^{1,2}$ have mean zero and variance one  such that  
	\begin{align}
		Ff_z(F) + {\rm\mathbf{1}}_{\{ F>z\}} \in\mathbb{D}^{1,2}\;\;\text{for all } z\in\R.
	\end{align}
	Then with $N\sim \mathcal{N}(0,1)$ and $\Gamma_0$ defined as in \eqref{Gamma0},  one has that
		\begin{align}
			d_K(F, N) \leq    \E \Big[ \big\vert   1  - \Gamma_0(F,  - L^{-1}F) \big\vert \Big] +  \sup_{z \in \R}
			\E \Big[ \big\langle D \big( Ff_z(F) + \1_{\{ F>z\}} \big)   \frac{D F}{\sqrt{p q}} ,  | D L^{-1} F | \big\rangle \Big].
	\end{align}
\end{theorem}
\begin{proof}
In the proof of \cite[Proposition 4.2]{DK19} we can find that $\PP (F \leq z) - \Phi(z) = \E \big[ f_z'(F) \big] - \E \big[ \Gamma_0 (f_z(F), - L^{-1} F) \big]$ for all $z\in\R$.
With the definition of $\Gamma_0$ we obtain
\begin{equation}
| \PP (F \leq z) - \Phi(z) | \leq  \E \Big[ \big\vert   1  - \Gamma_0(F,  - L^{-1}F) \big\vert \Big]  \ + 
\sum_{k \in \N} \E \Big[ \big\vert R_k \big\vert \times \big\vert D_k L^{-1} F   \big\vert \Big];
\end{equation}
for the arguments see the calculations which lead to \cite[Equation (100)]{DK19} with $R_k := D_k f_z(F) - f_z'(F) D_kF$, for every $k \in \N$.
We also used the independence of $X_k$ and $|R_k| |D_k L^{-1} F|$ to get this inequality starting from \cite[Equation (100)]{DK19}.
From this point on the result follows as in the proof of part (1) of Theorem  \ref{Kol_RAD}.
\end{proof}

\begin{remark}\label{rem_FMT} 
Now, let us  present a simple path that leads to the fourth-moment-influence bound in Kolmogorov distance \cite[Theorem 1.1]{DK19} already mentioned in the introduction. For the bound in Wasserstein distance, we refer interested readers to \cite{Zheng19} for a simple proof using exchangeable pairs. 
Suppose  that $F = J_m(f)\in L^4(\Omega)$ for some $f\in\fH^{\odot m}_0$ and $m\in\N$ such that $\E[ F^2]=1$, then it has been pointed out in the proof of  \cite[ Lemma 3.7]{DK19} that the random sequence $u = (u_k)_{k\in\N}$, defined as in \eqref{def_G_u_k}, satisfies condition (2.14) in \cite{KRT17}, which implies $u\in\text{Dom}(\delta)$. Moreover, Equation (2.15) in \cite{KRT17} holds and in our notation it reads as follows:
\[\E\big[ \langle DG, u \rangle \big] = \E\big[ G \delta(u) \big] = \E\Big[ \big(Ff_z(F) + \1_{\{  F>z\} } \big) \delta(u) \Big],
\]
where $G$ and $u$ are defined as in  \eqref{def_G_u_k}. It follows that 
\begin{align}\label{need2}
		\sum_{k\in\N} \E\Big[  | R_k| \times  | D_kL^{-1} F | \Big] \leq  2 \E\big[ |   \delta(u) | \big] \leq  2 \big\| \delta\big((pq)^{-1/2}DF | DL^{-1} F | \big) \big\|_{L^2(\Omega)}.
\end{align}
It is also proved in  \cite[Lemma 3.7]{DK19} that 
\begin{align}
\big\| \delta\big((pq)^{-1/2}DF | DF | \big) \big\|_{L^2(\Omega)}^2 \leq  (8m^2 - 7) \Bigg(  (4m-3) \big( \E[ F^4] -3\big) + (6m-3)\gamma_m \mathcal{M}(f) \Bigg), \label{need11}
\end{align}
where $\gamma_m: = 2(2m-1)! \sum_{r=1}^m r! \binom{m}{r}^2$ and $\mathcal{M}(f)$, the maximal influence of $f$, is defined by 
\[
\mathcal{M}(f) := \sup_{k\in \N}  \sum_{1\leq i_2 < ... < i_m<\infty} f^2(k, i_2, ... , i_m).
\]
With $F=J_m(F)$ and $-L^{-1}F = m^{-1} F$, we obtain  
\[
\E \Big[ \big\vert   1  - \Gamma_0(F,  - L^{-1}F) \big\vert \Big] \leq 
\sqrt{ \var\big( m^{-1} \Gamma_0(F, F) \big)  }.
\]
Lemma 3.5 in \cite{DK19} tells us that
\[
\var\big(m^{-1} \Gamma_0(F,F) \big) \leq \frac{(2m-1)^2}{4m^2} \left(  \E[ F^4] - 3 + \gamma_m \mathcal{M}(f) \right), 
\]
which, together with \eqref{need11}, implies the bound
\begin{align*}
d_K(F, N) \leq & \frac{2m-1 + 4\sqrt{(8m^2-7)(4m-3)  }}{2m} \sqrt{\big\vert \E[F^4] - 3 \big\vert} \\
&\qquad + \frac{2m-1 + 4\sqrt{(8m^2-7)(6m-3)\gamma_m  }}{2m} \sqrt{ \mathcal{M}(f)},
\end{align*}
which, up to the numerical constants only depending on $m$, coincides with that in \cite[Theorem 1.1]{DK19}.
\end{remark}

\section{A discrete  second-order Gaussian Poincar\'e inequality}\label{sec:2ndorderpoincare}

While the discrete gradient admits a natural interpretation as a difference operator and is thus easy to handle, this is not the case for the other discrete Malliavin operators such as $\delta$, $L$ or $L^{-1}$. It is thus desirable to have a bound for $d_K(F,N)$ just in terms of the discrete gradient and its iterate. A result of this type is known as a  second-order Gaussian Poincar\'e inequality and has first been established for functionals of Gaussian random variables in \cite{Chatterjee2ndorder}. This has later been extended to more general
functionals of Gaussian random fields \cite{NourdinPeccatiReinertPoincare} and also to the Poisson framework \cite{LRPY21,LPS} and {the} Rademacher setting \cite{NourdinPeccatiReinert}, the latter using smooth probability metrics. For the Kolmogorov distance, a discrete Gaussian second-order Poincar\'e inequality was derived in \cite{KRT17}. The next theorem is a simplified version of the main result in \cite{KRT17}, which removes several superfluous terms.

\begin{theorem}\label{thm:2ndOrderPoincare}
Let  $F\in\mathbb{D}^{1,2}$ have mean zero and variance one, {and}  define
\begin{align*}
	B_1:&=  \sum_{j,k,\ell\in\N} \sqrt{  \E\big[ (D_jF)^2 (D_kF)^2 \big] }  \sqrt{  \E\big[ (D_\ell D_jF)^2 (D_\ell D_kF)^2 \big] } \\
	B_2:&=\sum_{j,k,\ell\in\N} \frac{1}{p_\ell q_\ell}   \E\big[ (D_\ell D_jF)^2 (D_\ell D_kF)^2 \big]  \qquad\qquad   B_3:= \sum_{k\in\N} \frac{1}{p_kq_k} \E\big[ (D_kF)^4 \big]  \\
	B_4:&=\sum_{k,\ell \in\N} \frac{1}{p_kq_k}  \sqrt{ \E\big[ (D_kF)^4 \big]  } \sqrt{  \E\big[   (D_\ell D_kF)^4 \big]    }\qquad  B_5:= \sum_{k,\ell \in\N} \frac{1}{p_kq_k p_\ell q_\ell}  \E\big[   (D_\ell D_kF)^4 \big].
	\end{align*}
Then the following statements hold true.
\begin{itemize}
\item[(1)] If  condition \eqref{condition_1} is satisfied and $(pq)^{-1/2} DF | DL^{-1}F| \in\text{\rm Dom}(\delta)$, then with $N\sim \mathcal{N}(0,1)$,
\begin{align}\label{2nd_R1}
	d_{K}(F, N)  \leq \frac{\sqrt{15}}{2} \sqrt{ B_1} + \frac{\sqrt{3}}{2} \sqrt{B_2} + 2 \sqrt{B_3} + 2 \sqrt{6} \sqrt{B_4} + 2\sqrt{3} \sqrt{B_5}.
\end{align}
\item[(2)] If  condition \eqref{A1} is satisfied and we assume either   {\rm (a)} $\sum_{k\in\N} p_k q_k <\infty$ or {\rm (b)} $F$ depends only on first $m$ Rademacher random variables, then with $N\sim \mathcal{N}(0,1)$,
\begin{align}\label{2nd_R2}
	d_{K}(F, N)  \leq \frac{\sqrt{15}}{2} \sqrt{ B_1} + \frac{\sqrt{3}}{2} \sqrt{B_2} + 4\sqrt{\kappa} \sqrt{B_3},
	\end{align}
	where $\kappa =\sum_{k\in\N} p_k q_k$ in case {\rm (a)} and  $\kappa =\sum_{k=1}^m p_k q_k$ in case {\rm (b)}. 
\end{itemize}

\end{theorem}

\begin{remark}\label{rem4_dw}
	\begin{itemize}
\item[(i)] Compared to the bound in \cite[Theorem 4.1]{KRT17},  we remark that the third and the fourth term
\begin{align*}
	A_3 &:= \sum_k  (p_k q_k)^{-1/2} \E[| D_kF|^3],\\
	A_4 &:= \| F\|_{L^r(\Omega)} \sum_{k\in\N} \frac{1}{\sqrt{p_k q_k}} \| D_k F\|_{L^t(\Omega)}\| D_k F\|_{L^{2s}(\Omega)} ^2
\end{align*}
there do not appear in Theorem \ref{thm:2ndOrderPoincare}, while 
\begin{center}
$\sqrt{ 15 B_1/4}=A_1$, $\sqrt{3B_2/4}=A_2$, $2\sqrt{B_3}=A_5$, $\sqrt{24 B_4}=A_6$ and $\sqrt{12 B_5}=A_7$
\end{center}
 in the notation of \cite{KRT17}.
 Especially we were able to remove the fourth term $A_4$, which involves the parameters $r,s,t>1$ with $r^{-1}+s^{-1}+t^{-1}=1$. Note that   the second bound \eqref{2nd_R2} contains  only three terms $B_1, B_2$ and $B_3$  and is more useful when  $\kappa B_3$ has the same order as $B_1, B_2$ or is of smaller order than $B_1, B_2$.
\item[(ii)] In  \cite[Remark 3.2]{Zheng17}, a second-order Gaussian Poincar\'e inequality is stated in Wasserstein distance: for $F\in\mathbb{D}^{1,2}$ with mean zero and variance one, 
\begin{align}
d_{W}(F, N): = \sup_h \big|\E\big[ h(F) - h(N) \big]  \big|\leq  \sqrt{\frac{15}{2\pi}} \sqrt{B_1} +  \sqrt{\frac{3}{2\pi}} \sqrt{B_2}    + A_3, \label{2nd_dw}
\end{align}
where the supremum runs over all $1$-Lipschitz   functions from $\R$ to $\R$ and $N\sim \mathcal{N}(0,1)$. The term $A_3$ may produce worse bound than other terms, see for example Remark \ref{rem_5}.
	\end{itemize}
\end{remark}

\begin{proof}[Proof of Theorem \ref{thm:2ndOrderPoincare}]
The bound is a direct consequence of Theorem \ref{Kol_RAD} and the computations already carried out in \cite{KRT17}. In fact, the   term
\[
 \E\Big[ \big\vert   1  - \langle DF,  - DL^{-1}F\rangle \big\vert \Big]
\]
 in Theorem \ref{Kol_RAD} is bounded by ${\sqrt{15}\over 2}\sqrt{B_1}+{\sqrt{3}\over 2}\sqrt{B_2}$ according to \cite[Equation (4.6)]{KRT17}. For the second term in \eqref{Kol_R1},  the computations on \cite[Pages 1093-1094]{KRT17} yield the following  $L^2$-bound:
\begin{align}
\big\| \delta\big(  (pq)^{-1/2} DF | DL^{-1}F| \big) \big\|_{L^2(\Omega)} \leq 2 \sqrt{B_3} + 2 \sqrt{6} \sqrt{B_4} + 2\sqrt{3} \sqrt{B_5}.  \label{need22}
\end{align}
This gives us the the bound \eqref{2nd_R1}, while the bound \eqref{2nd_R2} is also immediate. 
\end{proof}

\section{Proofs I: Infinite weighted $2$-runs}\label{sec:2runs}

We consider a Rademacher functional $F$, which belongs to the sum $\C_1\oplus\C_2$ of the first two Rademacher chaoses, that is, $F = J_1(f) + J_2(g)$ with $f\in\fH$ and $g\in\fH^{\otimes 2}$.  For simplicity,  we only consider the case where $p_k=1/2$ for all $k\in\N$ and assume $f\in \fH$, $g\in\fH^{\odot 2}_0$ satisfy $\var(F)= \| f\|^2_{\fH} + 2 \| g\|^2_{\fH^{\otimes 2}} =1$. Then by the hypercontractivity property, $F\in L^4(\Omega)$ so that the same arguments as in  Remark \ref{rem_FMT} imply that the bound \eqref{need2} still holds true. As a consequence, with $N\sim \mathcal{N}(0,1)$,
		\begin{align*}
			d_K(F, N)& \leq \E\Big[ \big\vert 1 -  \langle DF, -DL^{-1}F \rangle \big\vert \Big] + 2  \big\| \delta\big(  (pq)^{-1/2} DF | DL^{-1}F| \big) \big\|_{L^2(\Omega)} \\
			&\leq  \sqrt{ \var\big(  \langle DF, -DL^{-1}F \rangle   \big)} + 2 \big(2 \sqrt{B_3} + 2 \sqrt{6} \sqrt{B_4} + 2\sqrt{3} \sqrt{B_5} \big)
		\end{align*}
		by \eqref{need22} and Theorem \ref{Kol_RAD}, where $B_3, B_4, B_5$ are introduced in Theorem \ref{thm:2ndOrderPoincare}. Since $F\in \C_1\oplus \C_2$, we can get an explicit expression for $ \langle DF, -DL^{-1}F \rangle $ by direct computation:
		\[
		\sqrt{ \var\big(  \langle DF, -DL^{-1}F \rangle   \big)}  \leq 2\sqrt{2} \| g\star^1_1 g \1_{\Delta_2} \|_{\fH^{\otimes 2}} + 3 \| f\star^1_1 g\|^2_{\fH}, 
		\]
		where $\star^1_1$ denotes the star-contraction,  see \cite[Page 1728]{NourdinPeccatiReinert} for more explanation.

		Now let us compute the terms $B_3, B_4, B_5$. First notice that 
		\[
		D_k F = f(k) + 2J_1\big(  g(k, \,\cdot\,) \big)
		\]
		and by the hypercontractivity property in this symmetric setting\footnote{This moment inequality can be proved  by using the multiplication formula in the symmetric setting (see \cite[Proposition 2.9]{NourdinPeccatiReinert}).  }, we can find some absolute constant $\theta>0$ such that
		\[
		\E\big[ (D_kF)^4 \big] \leq \theta \Big( \E\big[ (D_kF)^2 \big] \Big)^2 = \theta \left(   f^2(k) + 4 \sum_{i\in \N}  g^2(k,i) \right)^2 \leq 2\theta   f^4(k)  + 32 \theta\left(  \sum_{i\in \N}  g^2(k,i) \right)^2.
		\]
		It follows  that 
		\begin{align*}
			B_3&= 4\sum_{k\in\N}  \E\big[ (D_kF)^4 \big] \leq 8\theta \sum_{k\in\N}  f^4(k) +   128\theta \sum_{k\in\N}    \left(  \sum_{i\in \N}  g^2(k,i) \right)^2,  \\
			B_4&=16 \theta^{1/2} \sum_{k \in\N}  \left[ f^2(k) + 4 \sum_{i\in \N}  g^2(k,i)\right]   \left(   \sum_{\ell \in\N}  g^2(\ell,k)  \right)^2 \\
			&  =16 \theta^{1/2} \sum_{k\in\N} f^2(k)  \left(   \sum_{\ell \in\N}  g^2(\ell,k)  \right)^2 +  64 \theta^{1/2}  \sum_{k\in\N}   \left(   \sum_{\ell \in\N}  g^2(\ell,k)  \right)^3,   \\
			B_5&= 256\sum_{k,\ell \in\N}       g^4(\ell,k). 
		\end{align*}
		Hence, we arrive at
		\begin{equation}\label{eq:J1+J2}
			\begin{split}
d_K(F, N)&  \leq C\Bigg(  \| g\star^1_1 g \1_{\Delta_2} \|_{\fH^{\otimes 2}} +  \| f\star^1_1 g\|^2_{\fH}  + \left( \sum_{k\in\N}  f^4(k)\right)^{1/2}  + \left(\sum_{k,\ell \in\N}       g^4(\ell,k)  \right)^{1/2}  \\
&\qquad\qquad\qquad + \left[ \sum_{k\in\N} \big[ 1 +  f^2(k)  \big] \left(   \sum_{\ell \in\N}  g^2(\ell,k)  \right)^2 \right]^{1/2}   \Bigg)
			\end{split}
		\end{equation}
		with a suitable absolute constant $C>0$. We remark that a similar bound for a probability metric defined by twice differentiable functions has been obtained in \cite[Proposition 5.1]{NourdinPeccatiReinert}.
		

\medspace

Given the bound \eqref{eq:J1+J2} we can now present the proof of Theorem \ref{thm:2runs}. We only sketch the main arguments and refer to  \cite[Section 5.3]{NourdinPeccatiReinert} for more detailed computations.  

\begin{proof}[Proof of Theorem \ref{thm:2runs}] First, by introducing $X_i =2\xi_i -1$ for $i\in\Z$, we get a sequence of independent and identically distributed symmetric Rademacher random variables and we can rewrite 
	\[
	F_n: = \frac{G_n - \E G_n }{\sqrt{\var(G_n) }} = J_1(f) + J_2(g),
	\]
	with 
	\begin{align*}
		f &= \frac{1}{4\sqrt{  \var(G_n) } } \sum_{a\in\Z} \alpha_a^{(n)} \big(\1_{\{a\}} + \1_{\{a+1\} }  \big) , \\
		g &= \frac{1}{8\sqrt{  \var(G_n) } } \sum_{a\in\Z} \alpha_a^{(n)} \big(  \1_{\{a\}}\otimes \1_{\{a+1\}}  +  \1_{\{a+1\}}\otimes \1_{\{a\}}     \big).
	\end{align*}
	Note that although here the Rademacher random variables are indexed by $\Z$ instead of $\N$, our main results can be fully carried to this setting. From  \cite[Section 5.3]{NourdinPeccatiReinert}, we have
	\begin{align}\label{last_s1}
		\| g\star^1_1 g \1_{\Delta_2} \|_{ \ell^2(\Z)^{\otimes 2}}  + \| f\star^1_1 g \1_{\Delta_2} \|_{ \ell^2(\Z)  } +  \left( \sum_{k\in\N}  f^4(k)\right)^{1/2}    \leq  \frac{1}{\var(G_n)  } \left( \sum_{i\in\Z}  \big( \alpha_i^{(n)}  \big)^4     \right)^{1/2}. 
	\end{align}
	Note that  $f(k)^2 = \frac{1}{16 \var(G_n)}\big[ \big( \alpha_k^{(n)}\big)^2 +    \big( \alpha_{k-1}^{(n)}\big)^2   \big] \leq 1$, in view of the expression 
	\[
	\var(G_n) = \frac{3}{16} \sum_{i\in\Z} \big( \alpha^{(n)}_i \big)^2 +  \frac{1}{8} \sum_{i\in\Z} \alpha^{(n)}_i \alpha^{(n)}_{i+1},
	\]
	see Equation (5.56) in  \cite{NourdinPeccatiReinert}.
	In fact, it is not difficult to see that
	\begin{align}\label{var_G_n}
		\frac{1}{16} \sum_{i\in\Z} \big( \alpha^{(n)}_i \big)^2    \leq  \var(G_n)  \leq  \frac{5}{16} \sum_{i\in\Z} \big( \alpha^{(n)}_i \big)^2.
	\end{align}
	It remains to estimate 
	\[
	\left( \sum_{k,\ell \in\Z}       g^4(\ell,k)  \right)^{1/2} +  \left[ \sum_{k\in\Z} \left(   \sum_{\ell \in\Z}  g^2(\ell,k)  \right)^2 \right]^{1/2}. 
	\]
	From 
	\[
	g(k,\ell) =  \frac{1}{8\sqrt{  \var(G_n) } }   \Big(    \alpha_k^{(n)}  \1_{\{\ell = k+1\}}  +    \alpha_{k-1}^{(n)}    \1_{\{\ell = k-1\}}     \Big),
	\]
	we obtain by direct computations that 
	\[
	\sum_{k,\ell \in\Z}       g^4(\ell,k) =  \frac{2}{8^4   \var(G_n) ^2 }  \sum_{k\in\Z}   \big(    \alpha_k^{(n)}  \big)^4  
	\]
	and
	\[
	\sum_{k\in\Z} \left(   \sum_{\ell \in\Z}  g^2(\ell,k)  \right)^2 = \frac{1}{8^4   \var(G_n) ^2 } \sum_{k\in\Z}  \Big[   \big(    \alpha_k^{(n)}  \big)^2   +    \big(    \alpha_{k-1}^{(n)}  \big)^2 \Big]^2 \leq  \frac{4}{8^4   \var(G_n) ^2 }  \sum_{k\in\Z}   \big(    \alpha_k^{(n)}  \big)^4.
	\]
	These two estimates, together with \eqref{last_s1}, imply the first bound in \eqref{2run1}. The inequalities in \eqref{var_G_n} imply the second bound  as well.  \end{proof}

\section{Proofs II: Subgraph counts in the Erd\H{o}s-R\'enyi random graph}\label{sec:ApplSubgraphs}

 Given a fixed graph $\fixGraph$,   we are interested in the number $S$ of subgraphs of $\mathbf{G}(n,p)$ that  are isomorphic to $\fixGraph$.  
From \cite[Lemma 3.5]{JLRBook} it is known that\, textcolor{green}{as $n\to\infty$,}
\begin{align*}
	\sigma^2 &:= \var(S) \asymp \frac{q  n^{2\vertices{\fixGraph}}p^{2\edges{\fixGraph}}}{\psi},
\end{align*}
where we recall that $\vertices{\fixGraph}, \edges{\fixGraph}$ are the number of vertices {and} edges of the graph $\fixGraph$, respectively, and that $q=1-p$.
As in the introduction, we put $W := (S-\E [S])/ \sigma$ and recall from Theorem \ref{thm:subgraphs} that our goal is to prove that 
\begin{align}\label{bdd_thm_subgraphs}
d_K(W, N)  = \cO\Big( (q \psi)^{-1/2} \Big), 
\end{align}
where $N\sim \mathcal{N}(0,1)$ and $\psi := \min_H n^{\vertices{H}}p^{\edges{H}}$ with the minimum running over all subgraphs $H$ of $\fixGraph$ with at least one edge.

\begin{proof}[Proof of Theorem \ref{thm:subgraphs}]
We start by observing that we may assume without loss of generality that $\fixGraph$ has no isolated vertices. Indeed, if $\fixGraph$ has isolated vertices, let $\fixGraph_0 \subset \fixGraph$ be the subgraph obtained by removing these vertices.
Let $S$ and $S_0$ be the number of subgraphs of $\mathbf{G}(n,p)$ that are isomorphic to $\fixGraph$ and $\fixGraph_0$, respectively.
Every copy of $\fixGraph_0$ in $\mathbf{G}(n,p)$ can be completed to a copy of $\fixGraph$ in $c=\binom{n-\vertices{\fixGraph_0}}{\vertices{\fixGraph}-\vertices{\fixGraph_0}}$ different ways.
It follows that $S = c\, S_0$ and $\sigma = c \, \sigma_0$ with $\sigma^2 = \var (S)$ and $\sigma_0^2 = \var(S_0)$.
This yields the identity $W = (S-\E [S])/ \sigma = (S_0-\E [S_0])/ \sigma_0$. Now let $H$ be a subgraph of $\fixGraph$ or $\fixGraph_0$, which may or may not have isolated vertices.
Let $H_0 \subset H$ be the subgraph obtained by removing all isolated vertices from $H$.
Then $H_0$ is a subgraph of both, $\fixGraph$ and $\fixGraph_0$,
with $\vertices{H_0} \leq \vertices{H}$ and $\edges{H_0} = \edges{H}$.
Hence, $n^{\vertices{H_0}}p^{\edges{H_0}} \leq n^{\vertices{H}}p^{\edges{H}}$.
This yields $\psi(n,p,\fixGraph) = \psi(n,p,\fixGraph_0)$.
	
By the previous discussion we assume from now on that $\fixGraph$ has no isolated vertices.
In this setting,  each subgraph $H$ of $\mathbf{G}(n,p)$ that is isomorphic to $\fixGraph$ (denoted by $H\simeq \fixGraph$) is 
uniquely identified 
by its set of edges.
Let $E$ be the set of all possible edges in the complete graph on $n$ vertices and  identify   a set of edges $\gamma \subset E$ with the induced subgraph.
In particular, we will  write $\vertices{\gamma}$ and $\edges{\gamma}$ for the number of vertices and edges of the subgraph induced by $\gamma$, respectively.
Note that the number of edges $\edges{\gamma}$ (regarding $\gamma$ as a subgraph) is the same as the cardinality $\vert \gamma \vert$ (regarding $\gamma$ as a set of edges).
Now, we define  $M: = \{ \gamma \subset E: \gamma \simeq \fixGraph\}$ and $M_k: = \{ \gamma\in M : k\in \gamma\}$,
where the latter set  consists of all copies of $\fixGraph$ which contain a given edge $k \in E$. Also, 
we define 
 \[
 X_k := \1_{\{ \text{edge $k$ is present in  $\textbf{G}(n,p)$}   \}} -  \1_{\{ \text{edge $k$ is not present in  $\textbf{G}(n,p)$}   \}}
 \]
for each $k \in E$,
so that $\{ X_k: k \in E \}$ is a finite set of independent and identically distributed Rademacher random variables.

For every $\gamma \in M$,    let $I_\gamma$ be the centred indicator for the presence of the subgraph $\gamma$ in $\mathbf{G}(n,p)$.
Then, we can represent $W$ as
\begin{align*}
W = \frac{S-\E S}{\sigma}
= \frac{1}{\sigma} \sum_{\gamma \in M} I_\gamma
\qquad\text{with}\qquad
I_\gamma = \left( \prod_{k \in \gamma} \frac{X_k +1}{2} \right) - p^{\edges{\fixGraph}}.
\end{align*}
Hence, the random variable $W$  is a Rademacher functional based only on finitely many of the Rademacher random variables $X_k$. In particular, both conditions of Theorem \ref{Kol_RAD}-(2) are fulfilled, see Remark \ref{rem1}.
Thus, we get the upper bound
\begin{align*}
	d_K(W, N) \leq    \E\Big[ \big\vert   1  - \langle DW,  - DL^{-1}W\rangle \big\vert \Big] +  2 \Big\| \delta\big((pq)^{-1/2} DW| DL^{-1}W| \big)  \Big\|_{L^1(\Omega)}.
\end{align*}
By the Cauchy-Schwarz inequality the first summand can be bounded by
$\sqrt{ \var\big(  \langle DW, -DL^{-1}W \rangle   \big)}$.
Following the steps (4.9), (4.10), (4.11) in \cite{KRT17}, we can also bound the second summand.
Taken together, we get
\[
	d_K(W, N) \leq
	\sqrt{ C_1 }
	+  2 \sqrt{ C_2 }
	+  2 \sqrt{ C_3 }
	\]
with
\begin{align*}
	C_1 :={}& \var\big(  \langle DW, -DL^{-1}W \rangle   \big), \qquad  \quad C_2 :={} \frac{1}{pq} \sum_{k \in E} \E\big[ (D_k W)^4 \big], \\
	C_3 :={}& \frac{1}{pq} \sum_{k,\ell \in E} \E\bigg[ \Big( D_\ell \big( (D_k W)(D_kL^{-1}W) \big) \Big)^2 \bigg] .
\end{align*}
To study these quantities we use the decomposition of $W$ as a sum over the terms $\sigma^{-1} I_\gamma$,
as well as the bilinearity of the covariance and the linearity of $D_k$ and $L^{-1}$.
This gives 
\begin{align*}
C_1 &= \var\left(  \sum_{k \in E} (D_kW)(D_kL^{-1}W)   \right) \\
	&= \sum_{k,\ell \in E} \cov\Big(  (D_kW) D_kL^{-1}W   ,   (D_\ell W) D_\ell L^{-1}W   \Big) \\
	&= \frac{1}{\sigma^4} \sum_{k,\ell \in E} \sum_{\gamma_1,\gamma_2,\gamma_3,\gamma_4 \in M}
		\cov\Big(  (D_k I_{\gamma_1})  D_kL^{-1} I_{\gamma_2}   ,   (D_\ell I_{\gamma_3}) D_\ell L^{-1} I_{\gamma_4}   \Big),
\end{align*}
as well as 		
\begin{align*}
C_2 &= \frac{1}{pq} \frac{1}{\sigma^4} \sum_{k \in E} \sum_{\gamma_1,\gamma_2,\gamma_3,\gamma_4 \in M}  \E \left[ \prod_{j=1}^4 D_k I_{\gamma_j} \right],\\
C_3 &= \frac{1}{pq} \frac{1}{\sigma^4} \sum_{k,\ell \in E} \sum_{\gamma_1,\gamma_2,\gamma_3,\gamma_4 \in M}
		\E\bigg[
			\bigg( D_\ell \Big( (D_k I_{\gamma_1})  D_kL^{-1} I_{\gamma_2} \Big) \bigg)
			 D_\ell \Big( (D_k I_{\gamma_3}) D_kL^{-1} I_{\gamma_4}  \Big) 
		\bigg].
\end{align*}
As in  \eqref{Yk}, we write   $Y_k := \frac{X_k - p + q}{2\sqrt{pq}}$ for  $k\in E$ and get
\begin{align*}
I_\gamma
&=  \left( \prod_{k \in \gamma}   \big(  \sqrt{pq} \,  Y_k + p  \big)  \right) - p^{\edges{\fixGraph}} = \sum_{\varnothing \neq A \subset \gamma}   p^{\edges{\fixGraph} - {\frac{\vert A \vert}{2}}} q^{\frac{\vert A \vert}{2}}  
	\prod_{a \in A} Y_a.
	\end{align*}
Since   $D_k Y_\ell =\1_{\{ k =\ell\}}$, we have
\begin{align*}
D_k I_\gamma
&= \sum_{\varnothing \neq A \subset \gamma} \1_{\{k \in A\}}    p^{\edges{\fixGraph} - {\frac{\vert A \vert}{2}}} q^{\frac{\vert A \vert}{2}}   
	\prod_{a \in A \backslash \{ k \}} Y_a   
\end{align*}
for $k \in E$. Analogously,
\begin{align*}
 -  L^{-1}I_\gamma
&= \sum_{\varnothing \neq A \subset \gamma}   \frac{1}{\vert A \vert} p^{\edges{\fixGraph} - {\frac{\vert A \vert}{2}}} q^{\frac{\vert A \vert}{2}}   
	\prod_{a \in A} Y_a, \\
 -   D_kL^{-1}I_\gamma
&= \sum_{\varnothing \neq A \subset \gamma}  \frac{1}{\vert A \vert}   \1_{\{k \in A\}} p^{\edges{\fixGraph} - {\frac{\vert A \vert}{2}}} q^{\frac{\vert A \vert}{2}}   
	\prod_{a \in A \backslash \{ k \}} Y_a.
\end{align*}
Hence,  using the short-hand  notation  $Y_B := \prod_{b\in B} Y_b $, we can further  write 
\begin{align*}
C_1 &= \frac{p^{4\edges{\fixGraph}}}{\sigma^4} \sum_{k,\ell \in E} \sum_{\gamma_1,\gamma_2,\gamma_3,\gamma_4 \in M}
		\sum_{\substack{
			\varnothing \neq A_1 \subset \gamma_1 \\
			\varnothing \neq A_2 \subset \gamma_2 \\
			\varnothing \neq A_3 \subset \gamma_3 \\
			\varnothing \neq A_4 \subset \gamma_4 }}
		\frac{1}{\vert A_2 \vert}  \frac{1}{\vert A_4 \vert}  
		f_{1,k,\ell} (A_1, A_2, A_3, A_4), \\
C_2 &= \frac{p^{4\edges{\fixGraph}}}{\sigma^4} \sum_{k \in E} \sum_{\gamma_1,\gamma_2,\gamma_3,\gamma_4 \in M}
		\sum_{\substack{
			\varnothing \neq A_1 \subset \gamma_1 \\
			\varnothing \neq A_2 \subset \gamma_2 \\
			\varnothing \neq A_3 \subset \gamma_3 \\
			\varnothing \neq A_4 \subset \gamma_4 }}
		f_{2,k}(A_1,A_2,A_3,A_4), \\
C_3 &= \frac{p^{4\edges{\fixGraph}}}{\sigma^4} \sum_{k,\ell \in E} \sum_{\gamma_1,\gamma_2,\gamma_3,\gamma_4 \in M}
		\sum_{\substack{
			\varnothing \neq A_1 \subset \gamma_1 \\
			\varnothing \neq A_2 \subset \gamma_2 \\
			\varnothing \neq A_3 \subset \gamma_3 \\
			\varnothing \neq A_4 \subset \gamma_4 }}
		\frac{1}{\vert A_2 \vert}  \frac{1}{\vert A_4 \vert} 
		f_{3,k,\ell}(A_1,A_2,A_3,A_4),
\end{align*}
where
\begin{align*}
	f_{1,k,\ell} (A_1, A_2, A_3, A_4)
	:&= 	p^{-\frac{\vert A_1 \vert + \vert A_2 \vert + \vert A_3 \vert + \vert A_4 \vert}{2}}
		q^{\frac{\vert A_1 \vert + \vert A_2 \vert + \vert A_3 \vert + \vert A_4 \vert}{2}}  \1_{ \left\{   \substack{ k \in A_1 \cap A_2 \\  \ell \in A_3 \cap A_4}   \right\}}  \notag   \\
	&\hspace{2cm} \times 	\cov\big(
			  Y_{A_1 \backslash \{ k \}} 
			 Y_{ A_2 \backslash \{ k \}},
			 Y_{ A_3 \backslash \{ \ell \}} 
			 Y_{A_4 \backslash \{ \ell \}}
		\big)  \notag  \\
		& \quad  \notag  \\
	f_{2,k}(A_1,A_2,A_3,A_4)
	:&= 	p^{-\frac{\vert A_1 \vert + \vert A_2 \vert + \vert A_3 \vert + \vert A_4 \vert}{2} -1}
		q^{\frac{\vert A_1 \vert + \vert A_2 \vert + \vert A_3 \vert + \vert A_4 \vert}{2} -1}
		 \1_{\{ k \in A_1 \cap A_2 \cap A_3 \cap A_4\}}  \notag \\
		& \hspace{2cm}  \times  \E\big[
			 Y_{A_1 \backslash \{ k \}} 
			  Y_ {A_2 \backslash \{ k \}} 
			 Y_ {A_3 \backslash \{ k \}} 
			Y_{A_4 \backslash \{ k \}}   \big]   \notag \\
					& \quad \notag 
\end{align*}
and
\begin{align}
	f_{3,k,\ell} (A_1, A_2, A_3, A_4)
	:&= 	p^{-\frac{\vert A_1 \vert + \vert A_2 \vert + \vert A_3 \vert + \vert A_4 \vert}{2} -1}
		q^{\frac{\vert A_1 \vert + \vert A_2 \vert + \vert A_3 \vert + \vert A_4 \vert}{2} -1}
		 \1_{\{ k \in A_1 \cap A_2 \cap A_3 \cap A_4\}}  \notag \\
		&\hspace{2cm} \times \E\big[
			\big( D_\ell \big( 
				 Y_{A_1 \backslash \{ k \}} 
			  Y_ {A_2 \backslash \{ k \}}  
			\big) \big)
		 D_\ell \big( 
				 Y_{A_3 \backslash \{ k \}} 
			  Y_ {A_4 \backslash \{ k \}}  
			\big)  
		\big]. \label{in_there}
\end{align}
For given $k, \ell \in E$, $\gamma_i \in M$, $\varnothing \neq A_i \subset \gamma_i$, $i = 1,2,3,4$,
we put
\begin{center}
$h_1 := \gamma_1 \cap \gamma_2$, \quad
$h_2 := (\gamma_1 \cup \gamma_2) \cap \gamma_3$ \quad and \quad 
$h_3 := (\gamma_1 \cup \gamma_2 \cup \gamma_3) \cap \gamma_4$,
\end{center}
and claim that 
\begin{align}
	\vert f_{1,k,\ell} (A_1, A_2, A_3, A_4) \vert
		&\leq 
		2q^2 p^{- (\edges{h_1} + \edges{h_2} + \edges{h_3})} 
		\label{f1bound} \\
	\vert f_{2,k}(A_1,A_2,A_3,A_4) \vert
		&\leq 
		 q p^{- (\edges{h_1} + \edges{h_2} + \edges{h_3})}  
		\label{f2bound} \\
	\vert f_{3,k,\ell} (A_1, A_2, A_3, A_4) \vert
		&\leq 
		q  p^{- (\edges{h_1} + \edges{h_2} + \edges{h_3})}.
		\label{f3bound}
\end{align}
In what follows,  we will first  deduce  from  \eqref{f1bound}--\eqref{f3bound} the following bounds:
\begin{equation}\label{eq:19-07a}
C_1 = \cO\big(  \psi^{-1} \big), \qquad C_2 = \cO\big(  (q\psi)^{-1} \big) \quad{\rm and}\quad C_3=  \cO\big(  (q\psi)^{-1} \big),
\end{equation}
which in turn imply \eqref{bdd_thm_subgraphs}.  The verification of the claims \eqref{f1bound}--\eqref{f3bound} is postponed to the end of the current section. 

\medskip

\paragraph{\bf Bounding $C_1$.}   If $f_{1,k,\ell} (A_1, A_2, A_3, A_4) \neq 0$,  then
\begin{itemize}
\item[(i)]  $\gamma_1, \gamma_2 \in M_k$ and $\gamma_3, \gamma_4 \in M_\ell$, implying that
  $\edges{h_1}, \edges{h_3} \geq 1$;
  \item[(ii)]  $\gamma_1 \cup \gamma_2$ and $\gamma_3 \cup \gamma_4$ must have at least one edge in common, since otherwise $
			  Y_{A_1 \backslash \{ k \}} 
			 Y_{ A_2 \backslash \{ k \}}$ and
			$ Y_{ A_3 \backslash \{ \ell \}} 
			 Y_{A_4 \backslash \{ \ell \}}$ are independent (with thus vanishing covariance);
	\item[(iii)]  by changing the names of $\gamma_3$ and $\gamma_4$ (if necessary),   we can ensure
that $\gamma_1 \cup \gamma_2$ and $\gamma_3$ always have at least one common edge,
so that $\edges{h_2}  \geq 1$.
  \end{itemize}
Therefore,
\begin{align*}
C_1
	&\leq \frac{p^{4\edges{\fixGraph}}}{\sigma^4}
		\sum_{k,\ell \in E}
		\sum_{\gamma_1,\gamma_2 \in M_k}
		\sum_{\gamma_3,\gamma_4 \in M_\ell}
		\sum_{\substack{
			\varnothing \neq A_1 \subset \gamma_1 \\
			\varnothing \neq A_2 \subset \gamma_2 \\
			\varnothing \neq A_3 \subset \gamma_3 \\
			\varnothing \neq A_4 \subset \gamma_4 }}
		2 \cdot \1_{\{ (\gamma_1 \cup \gamma_2) \cap \gamma_3 \neq \varnothing \}} 
		\vert f_{1,k,\ell} (A_1, A_2, A_3, A_4) \vert,
\end{align*}
and we can decompose the sums over $\gamma_1,\gamma_2\in M_k$ and  $\gamma_3,\gamma_4 \in M_\ell$ as follows, using additionally the claim \eqref{f1bound}:
\begin{align*}
C_1
	&\leq \frac{p^{4\edges{\fixGraph}}}{\sigma^4}
		\sum_{k \in E}
		\sum_{\substack{ \gamma_1 \in M_k}}
		\sum_{\substack{ h_1 \subset \gamma_1 \\ \edges{h_1} \geq 1}}
		\sum_{\substack{ \gamma_2 \in M_k \\ \gamma_1 \cap \gamma_2 = h_1}}
		\sum_{\substack{ h_2 \subset \gamma_1 \cup \gamma_2 \\ \edges{h_2} \geq 1}}
		\sum_{\ell \in E}
		\sum_{\substack{ \gamma_3 \in M_\ell \\ (\gamma_1 \cup \gamma_2) \cap \gamma_3 = h_2 }}\\
		&\hspace{2cm}\times
		\sum_{\substack{ h_3 \subset \gamma_1 \cup \gamma_2 \cup \gamma_3 \\ \edges{h_3} \geq 1}}
		\sum_{\substack{ \gamma_4 \in M_\ell \\ (\gamma_1 \cup \gamma_2 \cup \gamma_3) \cap \gamma_4 = h_3}}
	2\cdot
	 (2^{\edges{\fixGraph}})^4 \cdot 
		2q^2 p^{ - (\edges{h_1} + \edges{h_2} + \edges{h_3})} .
\end{align*}
Now we have to count the number of summands in the above display.
\begin{itemize}
\item[(1)] The sum over $k \in E$ contains  $\cO( n^2)$ terms.

\item[(2)] Given $k \in E$,  there are at most $\cO( n^{\vertices{\fixGraph}-2} )$ possibilities to complete $k$ to a copy $\gamma_1$ of $\fixGraph$.

\item[(3)]  Given $h_1 \subset \gamma_1$,  there are $\cO(n^{\vertices{\fixGraph}-\vertices{h_1}})$ possibilities to choose $\gamma_2 \in M$
so that $\gamma_1 \cap \gamma_2 = h_1$.

\item[(4)] The sum over $\gamma_4 \in M$ can be treated the same way and contains $\cO\big( n^{\vertices{\fixGraph} - \vertices{h_3}} \big)$ terms.
\end{itemize}
It remains to look at the double sum $\sum_{\ell \in E} \sum_{\gamma_3 \in M_\ell , (\gamma_1 \cup \gamma_2) \cap \gamma_3 = h_2  }$ for  given $h_2, \gamma_1, \gamma_2$.
Here we have to distinguish three cases according to the relation between $\ell$ and $\gamma_1 \cup \gamma_2$.

\textbf{Case 1:}
  There are at most $(\vertices{\gamma_1 \cup \gamma_2})^2$ possibilities to choose $\ell \in E$
	such that both vertices of  the edge $\ell$ are also vertices of $\gamma_1 \cup \gamma_2$.
	In this case,  these two vertices have to be in $h_2$ and there are  $\cO ( n^{\vertices{\fixGraph} - \vertices{h_2}} )$ possibilities to choose $\gamma_3 \in M$
	so that $(\gamma_1 \cup \gamma_2) \cap \gamma_3 = h_2$ and $\ell \in \gamma_3$.

\textbf{Case 2:}
There are at most $n\vertices{\gamma_1 \cup \gamma_2} $ possibilities to choose $\ell \in E$
	such that exactly one vertex of the edge  $\ell$ is also a vertex of $\gamma_1 \cup \gamma_2$. In this case, 
	  this vertex of $\ell$ has to be in $h_2$ and  the other vertex of $\ell$ is not in $h_2$. Therefore, there are  $\cO ( n^{\vertices{\fixGraph} - (\vertices{h_2}+1)} )$ possibilities
	to choose $\gamma_3 \in M$ so that $(\gamma_1 \cup \gamma_2) \cap \gamma_3 = h_2$ and $\ell \in \gamma_3$.

\textbf{Case 3:}
 There are at most $n^2$ possibilities to choose $\ell \in E$ such that 
	 none of the vertices of the edge $\ell$ is   a vertex of $\gamma_1 \cup \gamma_2$.
	In this case,  none of these two vertices is in $h_2$ and there are  $\cO( n^{\vertices{\fixGraph} - (\vertices{h_2} + 2)} )$ possibilities to choose $\gamma_3 \in M$
	so that $(\gamma_1 \cup \gamma_2) \cap \gamma_3 = h_2$ and $\ell \in \gamma_3$.

In each of these cases,  the  double sum $\sum_{\ell \in E} \sum_{\gamma_3 \in M_\ell , (\gamma_1 \cup \gamma_2) \cap \gamma_3 = h_2  }$ runs over 
$\cO\big( n^{\vertices{\fixGraph} - \vertices{h_2}} \big)$ 
terms.
Hence,
\begin{align*}
C_1 &= \cO \Bigg(
	\frac{p^{4\edges{\fixGraph}}}{\sigma^4} \cdot
	n^2 \cdot
	n^{\vertices{\fixGraph}-2}
	\sum_{\substack{ h_1 \subset \fixGraph \\ \edges{h_1} \geq 1}}
	n^{\vertices{\fixGraph} - \vertices{h_1}}
	\sum_{\substack{ h_2 \subset \fixGraph \\ \edges{h_2} \geq 1}}
	n^{\vertices{\fixGraph} - \vertices{h_2}}
	\sum_{\substack{ h_3 \subset \fixGraph \\ \edges{h_3} \geq 1}}
	n^{\vertices{\fixGraph} - \vertices{h_3}} \cdot
	p^{ - (\edges{h_1} + \edges{h_2} + \edges{h_3})} \cdot q^2
	\Bigg)\\
&= \cO \Bigg(
	\frac{q^2 \cdot n^{4\vertices{\fixGraph}} p^{4\edges{\fixGraph}}}{\sigma^4} \cdot
	\left(
	\sum_{ h \subset \fixGraph : \edges{h} \geq 1}
	n^{-\vertices{h}} p^{-\edges{h}}
	\right)^3
	\Bigg) = \cO \Bigg(
	\frac{q^2 \cdot n^{4\vertices{\fixGraph}} p^{4\edges{\fixGraph}}}{\sigma^4 \cdot \psi^3}
	\Bigg).
\end{align*}
Since  $\sigma^2 \asymp \frac{q \cdot n^{2\vertices{\fixGraph}}p^{2\edges{\fixGraph}}}{\psi}$, we  find that
$
C_1  = \cO\left( \psi^{-1} \right).
$

\medskip

\paragraph{\bf Bounding $C_2$ and $C_3$.} Most steps for handling $C_2$ and $C_3$ are similar to the above arguments and for that reason we only present the ideas. Suppose  that $f_{2,k} (A_1, A_2, A_3, A_4) \neq 0$, then 
  $\gamma_1, \gamma_2, \gamma_3, \gamma_4 \in M_k$.
We may assume that $\edges{h_1}, \edges{h_2}, \edges{h_3} \geq 1$.
Counting the non-trivial summands leads to
\begin{align*}
C_2
	&= \cO \Bigg(
	\frac{p^{4\edges{\fixGraph}}}{\sigma^4} \cdot
	n^2 \cdot
	n^{\vertices{\fixGraph}-2}
	\sum_{\substack{ h_1 \subset \fixGraph \\ \edges{h_1} \geq 1}}
	n^{\vertices{\fixGraph} - \vertices{h_1}}
	\sum_{\substack{ h_2 \subset \fixGraph \\ \edges{h_2} \geq 1}}
	n^{\vertices{\fixGraph} - \vertices{h_2}}
	\sum_{\substack{ h_3 \subset \fixGraph \\ \edges{h_3} \geq 1}}
	n^{\vertices{\fixGraph} - \vertices{h_3}} \cdot
	p^{ - (\edges{h_1} + \edges{h_2} + \edges{h_3})} \cdot q
	\Bigg)\\
&= \cO \Bigg(
	\frac{q \cdot n^{4\vertices{\fixGraph}} p^{4\edges{\fixGraph}}}{\sigma^4 \cdot \psi^3}
	\Bigg) = \cO\Big( (q\psi)^{-1} \Big).
\end{align*}
Similarly, suppose that
$f_{3,k,\ell} (A_1, A_2, A_3, A_4) \neq 0$, then  
  $\gamma_1, \gamma_2, \gamma_3, \gamma_4 \in M_k$,
so that  we may assume that $\edges{h_1}, \edges{h_2}, \edges{h_3} \geq 1$.
It further implies that $\ell \in \gamma_1 \cup \gamma_2$ and $\ell \in \gamma_3 \cup \gamma_4$ and $k \neq \ell$.
Otherwise, the $D_\ell$-operation in \eqref{in_there} would vanish, meaning that
\[
 D_\ell \big(   Y_{A_1 \backslash \{ k \}}   Y_ {A_2 \backslash \{ k \}}   \big) 		 
 = 
 D_\ell \big(   Y_{A_3 \backslash \{ k \}}   Y_ {A_4 \backslash \{ k \}}   \big) 
  =
  0
\]
By changing the names of $\gamma_1$ and $\gamma_2$ (if necessary),  we can ensure
that $\ell \in \gamma_1$.
This leads to
\begin{align*}
C_3 &\leq \frac{p^{4\edges{\fixGraph}}}{\sigma^4}
		\sum_{k,\ell \in E}
		\sum_{\gamma_1 \in M_k \cap M_\ell}
		\sum_{\gamma_2, \gamma_3,\gamma_4 \in M_k}
		\sum_{\substack{
			\varnothing \neq A_1 \subset \gamma_1 \\
			\varnothing \neq A_2 \subset \gamma_2 \\
			\varnothing \neq A_3 \subset \gamma_3 \\
			\varnothing \neq A_4 \subset \gamma_4 }}
		2 \cdot
		\vert f_{3,k,\ell} (A_1, A_2, A_3, A_4) \vert\\
	&= \cO \Bigg(
	\frac{p^{4\edges{\fixGraph}}}{\sigma^4} \cdot
	n^{\vertices{\fixGraph}}
	\sum_{\substack{ h_1 \subset \fixGraph \\ \edges{h_1} \geq 1}}
	n^{\vertices{\fixGraph} - \vertices{h_1}}
	\sum_{\substack{ h_2 \subset \fixGraph \\ \edges{h_2} \geq 1}}
	n^{\vertices{\fixGraph} - \vertices{h_2}}
	\sum_{\substack{ h_3 \subset \fixGraph \\ \edges{h_3} \geq 1}}
	n^{\vertices{\fixGraph} - \vertices{h_3}} \cdot
	p^{ - (\edges{h_1} + \edges{h_2} + \edges{h_3})} \cdot q
	\Bigg)\\
&= \cO \Bigg(
	\frac{q \cdot n^{4\vertices{\fixGraph}} p^{4\edges{\fixGraph}}}{\sigma^4 \cdot \psi^3}
	\Bigg) =\cO\Big( (q\psi)^{-1} \Big).
\end{align*}
This completes the proof of \eqref{eq:19-07a} and it remains to verify  \eqref{f1bound}-\eqref{f3bound}. We begin with the following observation. For given $\tilde A_1, \tilde A_2, \tilde A_3, \tilde A_4 \subset E$ define
\begin{align*}
	B_k & = B_k(\tilde A_1, \tilde A_2, \tilde A_3, \tilde A_4)  \\
	&:= \big\{ e \in E \,:\, e \text{ is an element that appears in  exactly } k \text{ of the sets } \tilde A_1, \tilde A_2, \tilde A_3, \tilde A_4 \big\}
\end{align*}
for $k=1,2,3,4$, and note that $\vert \tilde A_1 \vert + \vert \tilde A_2 \vert + \vert \tilde A_3 \vert + \vert \tilde A_4 \vert = \vert B_1 \vert + 2\vert B_2 \vert + 3\vert B_3 \vert + 4\vert B_4 \vert$. Using the independence of the random variables   $(Y_k)_{k \in E}$,  we can thus write
\begin{align}
	\Big\vert \E\big[
	 Y_{\tilde A_1} Y_{\tilde A_2} Y_{\tilde A_3} Y_{\tilde A_4} 	\big] \Big\vert
	&= \Big\vert
		  \E\big[ Y_{B_1} \big] 
		  \E\big[ Y_{B_2}^2 \big] 
		  \E\big[ Y_{B_3}^3 \big] 
		 \E\big[ Y_{B_4}^4 \big]
	\Big\vert \notag \\
	&\leq   \1_{\{ B_1 = \varnothing \}}   
		(pq)^{-\frac{\vert B_3 \vert }{2}} 
		(pq)^{-\vert B_4 \vert}=   \1_{\{ B_1 = \varnothing \}} 
		(pq)^{-\frac{\vert B_3 \vert + 2\vert B_4 \vert}{2}} \label{in_here}
\end{align}
and
\begin{align*}
	\Big\vert \E\big[
		  Y_{ \tilde A_1}  Y_{ \tilde A_2}  
	\big]  \E\big[
		  Y_{ \tilde A_3}   Y_{ \tilde A_4}
	\big] \Big\vert
	&=	\1_{\{ \tilde A_1 = \tilde A_2 \}} 
		\1_{\{ \tilde A_3 = \tilde A_4 \}}
	\leq \1_{\{ B_1 = \varnothing \}} 
		(pq)^{-\frac{\vert B_3 \vert + 2\vert B_4 \vert}{2}}.
\end{align*}

\bigskip

\paragraph{\bf Proof of  \eqref{f1bound}}
Let us first verify  \eqref{f1bound} and  assume  that $k \in A_1 \cap A_2$ and $\ell \in A_3 \cap A_4$.
With $\tilde A_1 = A_1 \backslash\{k\}$, $\tilde A_2 = A_2 \backslash\{k\}$, $\tilde A_3 = A_3 \backslash\{\ell\}$ and $\tilde A_4 = A_4 \backslash\{\ell\}$ we write 
\begin{align*}
	\vert A_1 \vert + \vert A_2 \vert + \vert A_3 \vert + \vert A_4 \vert 
	= \vert \tilde A_1 \vert + \vert \tilde A_2 \vert + \vert \tilde A_3 \vert + \vert \tilde A_4 \vert + 4 
	= \vert B_1 \vert + 2\vert B_2 \vert + 3\vert B_3 \vert + 4\vert B_4 \vert + 4
\end{align*}
and
\begin{align*}
	&\vert f_{1,k,\ell} (A_1, A_2, A_3, A_4) \vert\\
		&\leq{} \phantom{+}
			p^{-\frac{\vert A_1 \vert + \vert A_2 \vert + \vert A_3 \vert + \vert A_4 \vert}{2}}
			q^{\frac{\vert A_1 \vert + \vert A_2 \vert + \vert A_3 \vert + \vert A_4 \vert}{2}}
			\cdot \Big\vert \E\big[
				  Y_{ \tilde A_1}     Y_{ \tilde A_2}  Y_{ \tilde A_3}  Y_{ \tilde A_4} 
				\big] \Big\vert\\
			&\qquad\qquad+
			p^{-\frac{\vert A_1 \vert + \vert A_2 \vert + \vert A_3 \vert + \vert A_4 \vert}{2}}
			q^{\frac{\vert A_1 \vert + \vert A_2 \vert + \vert A_3 \vert + \vert A_4 \vert}{2}}
			 \Big\vert \E\big[
		  Y_{ \tilde A_1}  Y_{ \tilde A_2}  
	\big] \cdot \E\big[
		  Y_{ \tilde A_3}   Y_{ \tilde A_4}
	\big] \Big\vert\\
		&\leq{}
		2 \cdot \1_{\{ B_1 = \varnothing \}} \cdot
			p^{-\frac{\vert A_1 \vert + \vert A_2 \vert + \vert A_3 \vert + \vert A_4 \vert}{2} -\frac{\vert B_3 \vert + 2\vert B_4 \vert}{2}}
			q^{\frac{\vert A_1 \vert + \vert A_2 \vert + \vert A_3 \vert + \vert A_4 \vert}{2} -\frac{\vert B_3 \vert + 2\vert B_4 \vert}{2}}\\
		&\leq{}
		2 \cdot 
			p^{-\vert B_2 \vert -2\vert B_3 \vert -3\vert B_4 \vert - 2}
			q^{\vert B_2 \vert + \vert B_3 \vert + \vert B_4 \vert + 2}.
\end{align*}
Finally, note that for $i\in\{2,3,4\}$,  any edge $e \in B_i$ is contained in at least $i-1$ of the subgraphs $h_1$, $h_2$, $h_3$.
However, due to the choice of $\tilde A_i$,
the presence of $k$ in $A_1 \cap A_2 \subset h_1$ is not taken into account in the definition of the $B_i$'s.
The same applies to the presence of $\ell$ in $A_3 \cap A_4 \subset h_3$.
Therefore, we find that
\begin{align*}
	\vert B_2 \vert + 2\vert B_3 \vert + 3\vert B_4 \vert
	\leq \edges{h_1\backslash\{k\}} + \edges{h_2} + \edges{h_3\backslash\{\ell\}},
\end{align*}
which yields  
\begin{align*}
	\vert B_2 \vert + 2\vert B_3 \vert + 3\vert B_4 \vert + 2
	\leq \edges{h_1} + \edges{h_2} + \edges{h_3}
\end{align*}
and thus proves \eqref{f1bound}.

\bigskip

\paragraph{\bf Proof of  \eqref{f2bound}} Choosing $\tilde A_i = A_i \backslash\{k\}$ with $k \in A_i$ for $i\in\{1,2,3,4\}$
and then using the same arguments as before,
we see that
\begin{align*}
	\vert f_{2,k} (A_1, A_2, A_3, A_4) \vert
		\leq{}&
			p^{-\vert B_2 \vert -2\vert B_3 \vert -3\vert B_4 \vert - 3}
			q^{\vert B_2 \vert + \vert B_3 \vert + \vert B_4 \vert + 1}
\end{align*}
and
\begin{align*}
	\vert B_2 \vert + 2\vert B_3 \vert + 3\vert B_4 \vert + 3
	\leq \edges{h_1} + \edges{h_2} + \edges{h_3}.
\end{align*}
This proves \eqref{f2bound}.

\bigskip

\paragraph{\bf Proof of  \eqref{f3bound}} To prove \eqref{f3bound} we use similar arguments. We first have to look at the difference operator
$
D_\ell (   Y_{A_1 \backslash \{ k \}} Y_{A_2 \backslash \{ k \}}  )
$.
\begin{itemize}
\item[(a)] If $\ell \notin  (A_1 \cup A_2) \backslash \{ k \}$,
	this expression is just zero.
\item[(b)] If $\ell\in  (A_1 \backslash \{ k \}) \Delta (A_2 \backslash \{ k \})$, where $\Delta$ denotes the symmetric difference of two sets,
	then
	\[
	D_\ell \big(   Y_{A_1 \backslash \{ k \}} Y_{A_2 \backslash \{ k \}}  \big) =  Y_{A_1 \backslash \{ k, \ell \}}  Y_{A_2 \backslash \{ k, \ell \}},
	\]
	since $D_k Y_k   = 1$.
\item[(c)] If $\ell \in (A_1 \cap A_2) \backslash \{ k \}$,
	then
	\[
	D_\ell \big(   Y_{A_1 \backslash \{ k \}} Y_{A_2 \backslash \{ k \}}  \big) 
	=
	 \frac{ q^2-p^2 }{ 4\sqrt{pq}} \cdot 
	Y_{A_1 \backslash \{ k, \ell \}}  
	Y_{A_2 \backslash \{ k, \ell \}},
	\]
	since $D_k (Y_k^2) = \sqrt{pq}\cdot \big( (Y_k^2)^+_k - (Y_k^2)^-_k \big) = \dfrac{ q^2-p^2 }{ 4\sqrt{pq}}$.
\end{itemize}
We can condense the above  three cases into the following equation:
\begin{align*}
	D_\ell \Big( 
		Y_{A_1 \backslash \{ k \}}  Y_{A_2 \backslash \{ k \}}  
		\Big)
	&= \xi(A_1,A_2) \cdot Y_{A_1 \backslash \{ k, \ell \}}  
	Y_{A_2 \backslash \{ k, \ell \}}
 \end{align*}
with $\xi(A_1,A_2)$ satisfying the estimate
\begin{align*}
	\vert \xi(A_1,A_2) \vert
	&\leq \1_{\{ \ell \in A_1 \cup A_2 \}}
	\1_{\{ k \neq \ell \}}
	(pq)^{- \frac{1}{2}  \1_{\{ \ell \in A_1 \cap A_2 \} }    }.
\end{align*}
Hence,
by choosing $\tilde A_i = A_i \backslash\{k,\ell\}$ for $i\in\{1,2,3,4\}$ subject to  the conditions $k \neq \ell$, $k \in A_1 \cap A_2 \cap A_3 \cap A_4$ and $\ell \in (A_1 \cup A_2)  \cap  (A_3 \cup A_4)$,
 we find that
\begin{align*}
	\vert A_1 \vert + \vert A_2 \vert + \vert A_3 \vert + \vert A_4 \vert 
={}& \vert \tilde A_1 \vert + \vert \tilde A_2 \vert + \vert \tilde A_3 \vert + \vert \tilde A_4 \vert + 6 + 
		\1_{\{ \ell \in A_1 \cap A_2 \}} + \1_{\{ \ell \in A_3 \cap A_4 \}}\\
={}& \vert B_1 \vert + 2\vert B_2 \vert + 3\vert B_3 \vert + 4\vert B_4 \vert + 6 + 
		\1_{\{ \ell \in A_1 \cap A_2 \}} + \1_{\{ \ell \in A_3 \cap A_4 \}}.
\end{align*}
Thus, we can deduce from \eqref{in_here} that
\begin{align*}
	&\vert f_{3,k,\ell} (A_1, A_2, A_3, A_4) \vert\\
		\leq{}& \phantom{+}
			p^{-\frac{\vert A_1 \vert + \vert A_2 \vert + \vert A_3 \vert + \vert A_4 \vert}{2} -1}
			q^{\frac{\vert A_1 \vert + \vert A_2 \vert + \vert A_3 \vert + \vert A_4 \vert}{2} -1}
			\cdot (pq)^{-\frac{1}{2} \cdot \1_{\{ \ell \in A_1 \cap A_2 \}} -\frac{1}{2}\cdot \1_{\{ \ell \in A_3 \cap A_4 \}} }  \Big\vert \E\big[
				Y_{\tilde A_1}  Y_{\tilde A_2}  Y_{\tilde A_3}  Y_{\tilde A_4}  
				\big] \Big\vert\\
		\leq{}&
			p^{-\vert B_2 \vert -2\vert B_3 \vert -3\vert B_4 \vert - 4 - \1_{\{ \ell \in A_1 \cap A_2 \}} - \1_{\{ \ell \in A_3 \cap A_4 \}}}
			q^{\vert B_2 \vert + \vert B_3 \vert + \vert B_4 \vert + 2}.
\end{align*}
Recall that for $i\in\{2,3,4\}$,  any edge $e \in B_i$ is contained in at least $i-1$ of the subgraphs $h_1$, $h_2$ and $h_3$.
Due to the choice of $\tilde A_i$, the presence of $k$ in $h_1$, $h_2$ and $h_3$ is not taken into account in the definition of the $B_i$'s.
The same applies to the possible presence of $\ell$ in $h_1$, $h_2$, $h_3$.
Thus, we have
\begin{align*}
	\vert B_2 \vert + 2\vert B_3 \vert + 3\vert B_4 \vert
	\leq \edges{h_1\backslash\{k,\ell\}} + \edges{h_2\backslash\{k,\ell\}} + \edges{h_3\backslash\{k,\ell\}},
\end{align*}
which yields 
 \begin{align*}
	\vert B_2 \vert + 2\vert B_3 \vert + 3\vert B_4 \vert + 3 + \1_{\{ \ell \in h_1 \}} + \1_{\{ \ell \in h_2 \}} + \1_{\{ \ell \in h_3 \}}
	\leq \edges{h_1} + \edges{h_2} + \edges{h_3}.
\end{align*}
The condition $\ell \in (A_1 \cup A_2)\cap (A_3 \cup A_4)$ implies  
\begin{center}
$\1_{\{ \ell \in h_1 \}} + \1_{\{ \ell \in h_2 \}} + \1_{\{ \ell \in h_3 \}} \geq 1 + \1_{\{ \ell \in A_1 \cap A_2 \}} + \1_{\{ \ell \in A_3 \cap A_4 \}}$.
\end{center}
Therefore,
\begin{align*}
	\vert B_2 \vert + 2\vert B_3 \vert + 3\vert B_4 \vert + 4 + \1_{\{ \ell \in A_1 \cap A_2 \}} + \1_{\{ \ell \in A_3 \cap A_4 \}}
	\leq \edges{h_1} + \edges{h_2} + \edges{h_3}.
\end{align*}
This proves \eqref{f3bound} and eventually completes the proof of Theorem \ref{thm:subgraphs}.
\end{proof}

\section{Proofs III: Vertices of fixed degree in the Erd\H{o}s-R\'enyi random graph}\label{sec:ApplERgraph}

For $d\in\N_0$, we are interested in the number $V_{d}$ of vertices of degree $d$ in the Erd\H{o}s-R\'enyi random graph $\textbf{G}(n,p)$. It is known from \cite[Chapter 6.3]{JLRBook} that
\[
\E[V_d] = n{n-1\choose d}p^d(1-p)^{n-d-1}
\]
and
\begin{align*}
	\var(V_d) &= \frac{n}{n-1} \binom{n-1}{d}^2 ((n-1)p-d)^2 p^{2d-1} (1-p)^{2n-2d-3} + \E[V_d] - \frac{1}{n}(\E[V_d])^2\,.
\end{align*}
{From now on we focus on the situation where $np=\cO(1)$, which is equivalent to $(1-p)^n\asymp 1$, as $n\to\infty$, since almost surely $V_d\leq n$.} Note that in this situation
\[
\E[V_d] \asymp n^{d+1}p^d .
\]  
For $d=0$, using $ (n-1)p (1-p)^{n-2} \leq   1-  (1-p)^{n-1}  \leq  (n-1)p  $, we get 
\begin{align}
\var(V_0)  &=  n (n-1)p (1-p)^{2n-3} +   n(1-p)^{n-1}   -n(1-p)^{2n-2}  \notag \\
 &\asymp n^2p. \label{lower_d=0}
\end{align}
For $d\geq 1$ the situation is as follows:
  \begin{itemize}
  \item If $np\to 0$ as $n\to\infty$, then by direct computations
  \begin{align}\label{asym:np_0}
\var(V_d)\asymp n^{d+1}p^d.
  \end{align}
  
  \item If $0< \liminf_{n\to\infty}np \leq \limsup_{n\to\infty} np< 1$, then by direct computations, 
  \begin{align}\label{asym:0_np_1}
\var(V_d)\asymp n.
  \end{align}

 \item If $np\to \lambda\in [ 1,\infty)$ as $n\to\infty$, then by direct computations
  \begin{equation} \label{asym:np=1}
  	\begin{split}
  \lim_{n\to\infty} \frac{\E[V_d]}{n} &= \frac{e^{-\lambda}\lambda^d}{d! },
  \\ \lim_{n\to\infty} \frac{\var[V_d]}{n} &=\frac{(d-\lambda)^2}{\lambda} \left( \frac{e^{-\lambda}\lambda^d}{d! }\right)^2    +  \frac{e^{-\lambda}\lambda^d}{d! }- \left( \frac{e^{-\lambda}\lambda^d}{d! }\right)^2.
  	\end{split}
  \end{equation}    
   \end{itemize}
   It follows that $\E[V_d] \asymp \var(V_d)\asymp n$ when both $np$ and $(np)^{-1}$ are of order $\cO(1)$. We can now give the proof of Theorem \ref{thm:VertexDegrees}.

\begin{proof}[Proof of Theorem \ref{thm:VertexDegrees}]
Let $e_1,\ldots,e_{n\choose 2}$ be an arbitrary labelling of the $n\choose 2$ edges of the complete graph on $\{1,\ldots,n\}$ and  we define 
 \[
 X_k := \1_{\{ \text{$e_k$ is present in  $\textbf{G}(n,p)$}   \}} -  \1_{\{ \text{$e_k$ is not present in  $\textbf{G}(n,p)$}   \}}.
 \]
This gives us a finite sequence of independent Rademacher random variables $\{ X_k: k =1, ... , \binom{n}{2}\}$ and the random variable $F_d$  is a Rademacher functional based on $X_k$'s.  Since $F_d\in L^\infty(\Omega)$ only depends on  finitely many  Rademacher random variables for each fixed $d\in\N_0$, we automatically have that $F_d \in\mathbb{D}^{1,2}$ and    $DF_d | DL^{-1} F_d| \in\text{Dom}(\delta)$. So, we are in the set-up of Theorem \ref{thm:2ndOrderPoincare} and  thus, we need to determine the asymptotic behaviour of the quantities $B_1,\ldots,B_5$ therein.   Here, we point out that the constant $\kappa$ satisfies
\[
\kappa = \sum_{k=1}^{\binom{n}{2}} pq \asymp n^2pq,
\]
which appears in the second bound \eqref{2nd_R2} in Theorem \ref{thm:2ndOrderPoincare}. In particular, we notice that if $n^2p\to\infty$, the term $\kappa B_3$ is worse (i.e., of larger order) than $B_3$ alone.

\begin{figure}[t]
	\begin{center}
		\begin{tikzpicture}
			\coordinate (A) at (0,0);
			\coordinate (B) at (2,1);
			\coordinate (C) at (-0.1,-0.2);
			\coordinate (D) at (0.1,-0.2);
			\coordinate (E) at (-0.2,0);
			\coordinate (F) at (2-0.1,1+0.2);
			\coordinate (G) at (2+0.1,1+0.2);
			\fill (A) circle (2pt);
			\fill (B) circle (2pt);
			\draw [dashed] (A) -- (B);
			\draw (A) -- (C);
			\draw (A) -- (D);
			\draw (A) -- (E);
			\draw (B) -- (F);
			\draw (B) -- (G);
			\node at (0.9,0.7) {$e_k$};
			\node at (0,-0.4) {${\rm deg}=d$};
			\node at (2,1.4) {${\rm deg}\neq d,d+1$};
			\node at (0.75,-1) {$(V_d)_k^+-(V_d)_k^-=2$};

			\coordinate (A) at (0+4,0);
			\coordinate (B) at (2+4,1);
			\coordinate (C) at (-0.1+4,-0.2);
			\coordinate (D) at (0.1+4,-0.2);
			\coordinate (E) at (-0.2+4,0);
			\coordinate (F) at (0.2+4,0);
			\coordinate (G) at (2-0.1+4,1+0.2);
			\coordinate (H) at (2+0.1+4,1+0.2);
			\fill (A) circle (2pt);
			\fill (B) circle (2pt);
			\draw [dashed] (A) -- (B);
			\draw (A) -- (C);
			\draw (A) -- (D);
			\draw (A) -- (E);
			\draw (A) -- (F);
			\draw (B) -- (G);
			\draw (B) -- (H);
			\node at (0.9+4,0.7) {$e_k$};
			\node at (0+4,-0.4) {${\rm deg}=d+1$};
			\node at (2+4,1.4) {${\rm deg}\neq d,d+1$};
			\node at (0.75+4,-1) {$(V_d)_k^+-(V_d)_k^-=-1$};
			
			\coordinate (A) at (0+8,0);
			\coordinate (B) at (2+8,1);
			\coordinate (C) at (-0.1+8,-0.2);
			\coordinate (D) at (0.1+8,-0.2);
			\coordinate (E) at (-0.2+8,0);
			\coordinate (F) at (2+8-0.1,1+0.2);
			\coordinate (G) at (2+8+0.1,1+0.2);
			\coordinate (H) at (2+8+0.2,1);
			\fill (A) circle (2pt);
			\fill (B) circle (2pt);
			\draw [dashed] (A) -- (B);
			\draw (A) -- (C);
			\draw (A) -- (D);
			\draw (A) -- (E);
			\draw (B) -- (F);
			\draw (B) -- (G);
			\draw (B) -- (H);
			\node at (0.9+8,0.7) {$e_k$};
			\node at (0+8,-0.4) {${\rm deg}=d$};
			\node at (2+8,1.4) {${\rm deg}=d$};
			\node at (0.75+8,-1) {$(V_d)_k^+-(V_d)_k^-=2$};

			\coordinate (A) at (0,0-3);
			\coordinate (B) at (2,1-3);
			\coordinate (C) at (-0.1,-0.2-3);
			\coordinate (D) at (0.1,-0.2-3);
			\coordinate (E) at (-0.2,0-3);
			\coordinate (E1) at (0.2,0-3);
			\coordinate (F) at (2-0.1,1+0.2-3);
			\coordinate (G) at (2+0.1,1+0.2-3);
			\coordinate (H) at (2+0.2,1-3);
			\coordinate (H1) at (2-0.2,1-3);
			\fill (A) circle (2pt);
			\fill (B) circle (2pt);
			\draw [dashed] (A) -- (B);
			\draw (A) -- (C);
			\draw (A) -- (D);
			\draw (A) -- (E);
			\draw (A) -- (E1);
			\draw (B) -- (F);
			\draw (B) -- (G);
			\draw (B) -- (H);
			\draw (B) -- (H1);
			\node at (0.9,0.7-3) {$e_k$};
			\node at (0,-0.4-3) {${\rm deg}=d+1$};
			\node at (2,1.4-3) {${\rm deg}=d+1$};
			\node at (0.75,-1-3) {$(V_d)_k^+-(V_d)_k^-=-2$};
			
			\coordinate (A) at (0+4,0-3);
			\coordinate (B) at (2+4,1-3);
			\coordinate (C) at (-0.1+4,-0.2-3);
			\coordinate (D) at (0.1+4,-0.2-3);
			\coordinate (E) at (-0.2+4,0-3);
			\coordinate (F) at (2-0.1+4,1+0.2-3);
			\coordinate (G) at (2+0.1+4,1+0.2-3);
			\coordinate (H) at (2+0.2+4,1-3);
			\coordinate (H1) at (2-0.2+4,1-3);
			\fill (A) circle (2pt);
			\fill (B) circle (2pt);
			\draw [dashed] (A) -- (B);
			\draw (A) -- (C);
			\draw (A) -- (D);
			\draw (A) -- (E);
			\draw (B) -- (F);
			\draw (B) -- (G);
			\draw (B) -- (H);
			\draw (B) -- (H1);
			\node at (0.9+4,0.7-3) {$e_k$};
			\node at (0+4,-0.4-3) {${\rm deg}=d$};
			\node at (2+4,1.4-3) {${\rm deg}=d+1$};
			\node at (0.75+4,-1-3) {$(V_d)_k^+-(V_d)_k^-=0$};
			
			\coordinate (A) at (0+8,0-3);
			\coordinate (B) at (2+8,1-3);
			\coordinate (C) at (-0.1+8,-0.2-3);
			\coordinate (D) at (0.1+8,-0.2-3);
			\coordinate (E) at (2-0.1+8,1+0.2-3);
			\coordinate (F) at (2+0.1+8,1+0.2-3);
			\fill (A) circle (2pt);
			\fill (B) circle (2pt);
			\draw [dashed] (A) -- (B);
			\draw (A) -- (C);
			\draw (A) -- (D);
			\draw (B) -- (E);
			\draw (B) -- (F);
			\node at (0.9+8,0.7-3) {$e_k$};
			\node at (0+8,-0.4-3) {${\rm deg}\neq d,d+1$};
			\node at (2+8,1.4-3) {${\rm deg}\neq d,d+1$};
			\node at (0.75+8,-1-3) {$(V_d)_k^+-(V_d)_k^-=0$};
		\end{tikzpicture}
	\end{center}
	\caption{The six possible cases for $(V_d)_k^+-(V_d)^-_k$. Here, ${\rm deg}=d$, ${\rm deg}=d+1$ or ${\rm deg}\neq d,d+1$ means that the vertex has degree $d$, $d+1$ or degree $\notin\{d,d+1\}$, respectively. In our illustration we have chosen $d=3$.}
	\label{fig:3casesIsolatedVertices}
\end{figure}
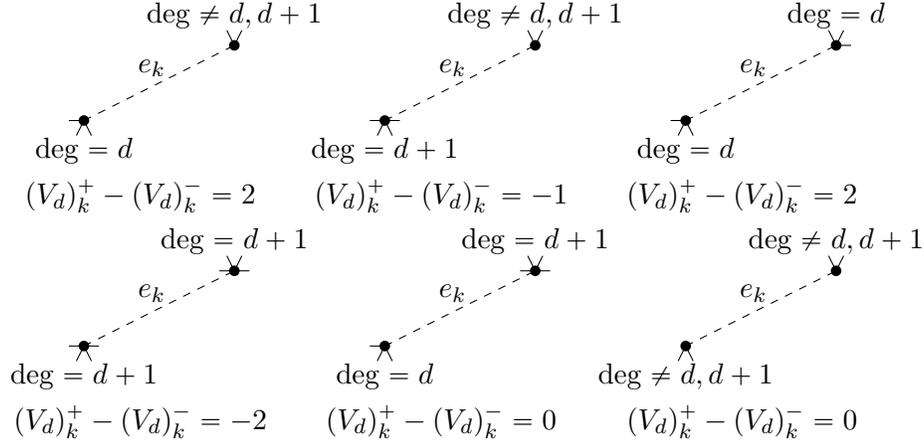
 
 \medskip
 
We first control the random variables $D_kF$ and $D_\ell D_kF$, for every $k,\ell=1,\dotsc,\binom{n}{2}$. We start with the first-order discrete gradient $D_kF$. For $k=1,\dotsc,\binom{n}{2}$, we have that
\begin{align*}
	D_kF_d = \sqrt{pq}((F_d)_k^+-(F_d)_k^-) = \frac{\sqrt{pq}}{\sqrt{\var(V_d)}}((V_d)_k^+-(V_d)_k^-)\,.
\end{align*}
Note that, for every $k=1,\dotsc,\binom{n}{2}$, $(V_d)_k^+$ equals the number of vertices of degree $d$ in $\textbf{G}(n,p)$ when $e_k$ belongs to $\textbf{G}(n,p)$, while $(V_d)_k^-$ equals the number of vertices of degree $d$ in $\textbf{G}(n,p)$ when $e_k$ does not belong to $\textbf{G}(n,p)$. Now, adding or removing an edge in $\textbf{G}(n,p)$ can both result in an increase or decrease of the total amount of vertices of degree $d$. For example, if $e_k$ connects two vertices of degree $d$, removing $e_k$ would set the counter down by 2. If $e_k$ connects two vertices of degree $d+1$, removing $e_k$ would set the counter up by 2. We thus have to distinguish the following six cases, see Figure \ref{fig:3casesIsolatedVertices}:
\begin{itemize}
	\item[-] \textit{Case 1.} $e_k$ connects a vertex of degree $d$ with a vertex of degree $\neq d,d+1$. In this case, $$(V_d)_k^+-(V_d)_k^- = 1\,.$$
	\item[-] \textit{Case 2.} $e_k$ connects a vertex of degree $d+1$ with a vertex of degree $\neq d,d+1$. In this case, $$(V_d)_k^+-(V_d)_k^- = -1\,.$$ 
	\item[-] \textit{Case 3.} $e_k$ connects two vertices of degree $d$. In this case, $$(V_d)_k^+-(V_d)_k^- = 2\,.$$
	\item[-] \textit{Case 4.} $e_k$ connects two vertices of degree $d+1$. In this case, $$(V_d)_k^+-(V_d)_k^- = -2\,.$$
	\item[-] \textit{Case 5.} $e_k$ connects a vertex of degree $d$ with a vertex of degree $d+1$. In this case, $$(V_d)_k^+-(V_d)_k^- = 0\,.$$
	\item[-] \textit{Case 6.} $e_k$ connects two vertices of degree $\neq d,d+1$. In this case, $$(V_d)_k^+-(V_d)_k^- = 0\,.$$
\end{itemize}
We emphasize that the cases 1, 3 and 5 do not occur if we consider the isolated vertex counting statistic $V_0$. Thus, by the above discussion we have that, for every $k=1,\dotsc,\binom{n}{2}$,
\begin{align}\label{Vertex degree count D_k bound 1}
	|D_kF_d| \leq \frac{2\sqrt{pq}}{\sqrt{\var(V_d)}}\,.
\end{align}
Next, we analyze the second-order discrete gradient $D_\ell D_kF_d$ for all $k,\ell=1, \dotsc, \binom{n}{2}$:
\begin{itemize}
\item  For $k = \ell$, we have that $D_\ell D_kF_d = 0$.
\item  For $k \neq \ell$, we have that
\begin{align*}
	D_\ell D_kF_d &= \frac{\sqrt{pq}}{\sqrt{\var(V_d)}}D_\ell\big(  (V_d)_k^+-(V_d)_k^- \big)  \\
	&=  \frac{\sqrt{pq}}{\sqrt{\var(V_d)}}D_\ell((V_d)_k^+-(V_d)_k^-)\1_{\{ | e_k \cap e_\ell| =1 \}},
\end{align*}
where   $| e_k \cap e_\ell| =1 $ means the two edges $e_k, e_\ell$ have only one common endpoint and the last equality follows from the fact that the difference  $(V_d)_k^+-(V_d)_k^-$ does not depend on the edge $e_\ell$. 
\end{itemize}
 By the same discussion as in  Figure \ref{fig:3casesIsolatedVertices},  we have that  $| D_\ell (V_d)_k^+| \leq 2\sqrt{pq}$ and  $| D_\ell (V_d)_k^-| \leq 2\sqrt{pq}$. 
Therefore,  
\begin{align}
	|{D_\ell D_kF_d}| &\leq   \frac{\sqrt{pq}}{\sqrt{\var(V_d)}}  \Big(\big\vert D_\ell (V_d)_k^+\big\vert  +    \big\vert  D_\ell (V_d)_k^- \big\vert   \Big)\1_{\{ | e_k \cap e_\ell| =1 \}}  
	\leq \frac{4pq}{\sqrt{\var(V_d)}} {\bf 1}_{\{\absolute{e_k \cap e_\ell}=1\}}.\label{Vertex degree count D_k bound 2}
\end{align}

We are now ready  to derive estimates for each of the terms $B_1,\ldots,B_5$.  By using  \eqref{Vertex degree count D_k bound 1} and \eqref{Vertex degree count D_k bound 2}, we get
\begin{align*}
	B_1 &= \sum_{j,k,\ell = 1}^{\binom{n}{2}} \Big(\E\big[(D_jF_d)^2(D_kF_d)^2\big] \Big)^{1/2} \Big(\E\big[(D_\ell D_jF_d)^2(D_\ell D_kF_d)^2\big]\Big)^{1/2} \notag\\
	&\leq \frac{64(pq)^3}{\big(\var(V_d)\big)^2} \sum_{j,k,\ell = 1}^{\binom{n}{2}} {\bf 1}_{\{ \absolute{e_j \cap e_\ell}=1 \}} {\bf 1}_{\{ \absolute{e_k \cap e_\ell}=1 \}} = \frac{64(pq)^3}{\big(\var(V_d)\big)^2} \binom{n}{2}(2(n-2))^2 \asymp \frac{(pq)^3n^4}{\big(\var(V_d)\big)^2}.  
\end{align*}
Similarly,  we get that
\begin{align*}
	B_2 &= \frac{1}{pq} \sum_{j,k,\ell = 1}^{\binom{n}{2}} \E\big[(D_\ell D_jF_d)^2(D_\ell D_kF_d)^2 \big] \leq \frac{256(pq)^3}{\big(\var(V_d)\big)^2} \sum_{j,k,\ell = 1}^{\binom{n}{2}} {\bf 1}_{\{ \absolute{e_j \cap e_\ell}=1 \}} {\bf 1}_{\{ \absolute{e_k \cap e_\ell}=1 \}} \\
	&= \frac{256(pq)^3}{\big(\var(V_d)\big)^2} \binom{n}{2}(2(n-2))^2 \asymp \frac{(pq)^3n^4}{(\var(V_d))^2},
\end{align*}
and  
 \begin{align*}
	B_3 = \frac{1}{pq} \sum_{k=1}^{\binom{n}{2}} \E[(D_kF_d)^4] \leq \frac{16pq}{(\var(V_d))^2} \binom{n}{2} \asymp \frac{pq n^2}{(\var(V_d))^2}. 
\end{align*}
Recall the second bound \eqref{2nd_R2} in Theorem \ref{thm:2ndOrderPoincare} and note that  the coefficient $\kappa$ satisfies $\kappa = \binom{n}{2} pq \asymp n^2 pq$. Then 
\[
\kappa B_3 \leq   \frac{p^2q^2 n^4}{(\var(V_d))^2}.
\]
It remains to estimate $B_4$ and $B_5$:
\begin{align*}
	B_4 &= \frac{1}{pq} \sum_{k,\ell = 1}^{\binom{n}{2}} (\E[(D_kF_d)^4])^{1/2} (\E[(D_\ell D_kF_d)^4])^{1/2} \leq \frac{64(pq)^2}{(\var(V_d))^2} \sum_{k,\ell = 1}^{\binom{n}{2}} {\bf 1}_{\{ \absolute{e_k \cap e_\ell}=1 \}}\\
	&= \frac{64(pq)^2}{(\var(V_d))^2} \binom{n}{2}2(n-2) \asymp \frac{(pq)^2}{(\var(V_d))^2} n^3\notag
\end{align*}
and
\begin{align*}
	B_5 &= \frac{1}{(pq)^2} \sum_{k,\ell=1}^{\binom{n}{2}} \E[(D_\ell D_k F_d)^4] \leq \frac{256(pq)^2}{(\var(V_d))^2} \sum_{k,\ell=1}^{\binom{n}{2}} 
	{\bf 1}_{\{ \absolute{e_k \cap e_\ell}=1 \}} \\
	&= \frac{256(pq)^2}{(\var(V_d))^2} \binom{n}{2}2(n-2) 
		\asymp \frac{(pq)^2}{(\var(V_d))^2} n^3.
\end{align*}
Now, let us prove the bound \eqref{Kol_F_0} in the regime $np = \cO(1)$ and $n^2p \to\infty$, as $n\to\infty$. Using the above estimates and the   bound \eqref{lower_d=0},  we get
\[
d_K(F_0, N) = \cO(p^{1/2} + n^{-1/2} + n^{-1} p^{-1/2}) = \cO(n^{-1} p^{-1/2}).
\]
Now, let us turn to the case $d\geq 1$. When $np\to 0$, $\var(V_d) \asymp n^{d+1} p^d$ and  $n^4p^3 + n^3p^2  = \cO(n^2p)$, then
\begin{align*}
d_K( F_d, N) =\cO \left(   \frac{ n^4p^3 + n^2p + n^3 p^2 }{ n^{2d+2} p^{2d}} \right)^{1/2} = \cO\big( p^{1/2} (np)^{-d}  \big).
\end{align*}
On the other hand, if $np =\cO(1)$ and $\liminf_{n\to\infty} np > 0$, we have that $\var(V_d) \asymp n$ and $n^4p^3 \asymp n^2p \asymp n^3 p^2 \asymp n$ so that
\begin{align*}
d_K( F_d, N) =\cO \left(   \frac{ n^4p^3 + n^2p + n^3 p^2 }{ n^2} \right)^{1/2} = \cO\big( n^{-1/2} \big).
\end{align*}
The proof is thus complete.
\end{proof}

\begin{remark}\label{rem_5} Recall the definition of $A_3$ from Remark  \ref{rem4_dw}. In the setting of Theorem \ref{thm:VertexDegrees}, we have that
\begin{align*}
A_3 =  (pq)^{-1/2} \sum_{k=1}^{\binom{n}{2}} \E\big[  | D_k F_d |^3 \big] \leq  \frac{ 8pq }{ \var(V_d)^{3/2}  }  \binom{n}{2}
 \asymp 
 \begin{cases}
 (n^2p)^{-1/2} &\text{if $d=0$} \\
 n^{-1/2}  &\text{if $d\geq 1$ and ${\displaystyle\liminf_{n\to\infty} np >0}$} \\
 n^{\frac{1}{2} - \frac{3d}{2}} p^{1-\frac{3d}{2}}  &\text{if $d\geq 1$ and $np \to 0$}.
 \end{cases}\end{align*}
As a consequence,  we get a Wasserstein bound for the random variables $F_d$ with the same order if $d\in\{0,1\}$ or if $d\geq 2$ and additionally ${\displaystyle\liminf_{n\to\infty} np >0}$, while the Wasserstein bound for   $d\geq 2$ in the case that $np\to 0$ is worse than our Kolmogorov bound, since 
\[
\frac{ n^{\frac{1}{2} - \frac{3d}{2}} p^{1-\frac{3d}{2}} }{ p^{1/2} (np)^{-d}   } = (np)^{\frac{1-d}{2} }\to \infty.
\]

\end{remark}

\section{Proofs IV: Isolated faces in the random $\kk$-complex}\label{sec:Complexes}

We recall that $I$ stands for the number of isolated $(\kk-1)$-faces in the random $\kk$-complex $\mathbf{Y}_\kk(n,p)$ and that $F=(I-\E I)/\sqrt{\var(I)}$. To see that $F$ is a Rademacher functional, we let $\big\{f_j: j=1,   \ldots, \binom{n}{ \kk+1} \big\}$ be an arbitrary labelling of the $\binom{n}{\kk+1}$ $\kk$-faces of an $(n-1)$-dimensional simplex  with $n$ vertices and put
\[
X_j := {\bf 1}_{\{f_j\text{ is present in }{\bf Y}_\kk(n,p)\}} - {\bf 1}_{\{f_j\text{ is not present in }{\bf Y}_\kk(n,p)\}}, j=1, ... , \binom{n}{ \kk+1},
\]
which are  i.i.d.  Rademacher random variables.

Clearly, $I$ (and hence $F$) is a functional over $\big\{X_j : j=1,\ldots,\binom{n}{\kk} \big\}$. Since $F\in L^\infty(\Omega)$ only depends on finitely many Rademacher random variables, all conditions of Theorem \ref{thm:2ndOrderPoincare} are automatically satisfied. 
We start our analysis by observing that
\begin{align*}
	\E\big[  I  \big] = {n\choose \kk}(1-p)^{n-\kk}.
\end{align*}
Now, let us  determine the variance $\var(I)$.  We denote by $\cF_{\kk-1}$   the set of $(\kk-1)$-faces of the $(n-1)$-dimensional simplex   and we can  represent $I$ as 
\[
I  = \sum_{f\in\cF_{\kk-1}}{\bf 1}_{\{f\text{ is isolated}\}}.
\] 
From now on, we only consider the case $\kappa\geq 2$, since the case $\kappa=1$ reduces to the setting in part (a) of  Theorem \ref{thm:VertexDegrees}.
The estimation of  $\var(I)$ begins with the following expression:
\[
\var(I) = \sum_{f\in\cF_{\kk-1}}\var({\bf 1}_{\{f\text{ is isolated}\}}) +  \sum_{f,g\in\cF_{\kk-1}  }\cov({\bf 1}_{\{f\text{ is isolated}\}},{\bf 1}_{\{g\text{ is isolated}\}}) \1_{\{ f\neq g\}}.
\]
Clearly, $\var({\bf 1}_{\{f\text{ is isolated}\}})=(1-p)^{n-\kk}(1-(1-p)^{n-\kk})$ for $f\in\cF_{\kk-1}$. Moreover,   the covariance in the second sum is non-zero   if and only if $f$ and $g$ share a common $(\kk-2)$-face;  in this case      
\begin{align*}
\E\big(  {\bf 1}_{\{f\text{ is isolated}\}}{\bf 1}_{\{g\text{ is isolated}\}} \big) = (1-p)^{n-\kk-1}(1-p)^{n-\kk-1}(1-p) = (1-p)^{2(n-\kk)-1}
\end{align*}
and hence
\begin{align*}
\var(I) &= {n\choose \kk}(1-p)^{n-\kk}(1-(1-p)^{n-\kk}) + {n\choose \kk-1}{n-\kk+1\choose 2}p(1-p)^{2(n-\kk)-1}.
\end{align*}
Under the assumptions of Theorem \ref{thm:Isolated faces}, this yields   $\var(I)\asymp n^{\kk+1}p$, as $n\to\infty$.

\begin{proof}[Proof of Theorem \ref{thm:Isolated faces}]
As already noted above, all assumptions of Theorem \ref{thm:2ndOrderPoincare} are satisfied and it remains to bound the terms $B_1,\ldots,B_5$. To do so, we start by controlling the discrete gradient $D_kF$ for $k=1,\ldots,{n\choose \kk+1}$. By definition,
\[
D_kF = {\sqrt{pq}\over\sqrt{\var(I)}}\big(I_k^+-I_k^-\big). 
\]
By construction of the random $\kappa$-complex, we have that $|I_k^+-I_k^-|\leq \kk+1$ almost surely, implying the bound
\[
|D_kF| \leq {(\kk+1)\sqrt{pq}\over\sqrt{\var(I)}}.
\]
Furthermore, the iterated discrete gradient $D_\ell D_kF$, $k,\ell=1,\ldots,{n\choose \kk+1}$ can be non-zero only when  $f_k$ and $f_\ell$  share a common $(\kk-1)$-face. So,
\[
|D_\ell D_kF| \leq {2(\kk+1)pq\over\sqrt{\var(I)}}{\bf 1}_{\{f_k\text{ and }f_\ell\text{ share a common $(\kk-1)$-face}\}}
\]
by the triangle inequality and the bound for the first-order discrete gradient.

We can now bound the five terms in Theorem \ref{thm:2ndOrderPoincare}. First, we obtain
\begin{align*}
B_1 &\leq {4(\kk+1)^4(pq)^3\over(\var(I))^2}\sum_{j,k,\ell=1}^{n\choose \kk+1}{\bf 1}_{\{f_j\text{ and }f_\ell\text{ share a common $(\kk-1)$-face}\}}\\
&\hspace{5cm}\times{\bf 1}_{\{f_k\text{ and }f_\ell\text{ share a common $(\kk-1)$-face}\}}\\
&=\cO\Big({p^3\over(n^{\kk+1}p)^2}n^{\kk+1}n^2\Big) = \cO(pn^{-\kk+1}),
\end{align*}
and we can obtain the same bound   for $B_2$.  For $B_3$, we have
\begin{align*}
B_3 &\leq {(\kk+1)^4pq\over(\var(I))^2}{n\choose \kk+1} = \cO\Big({p\over(n^{\kk+1}p)^2}n^{\kk+1}\Big) = \cO(n^{-(\kk+1)}p^{-1}).
\end{align*}
For $B_4$,  we deduce the inequality
\begin{align*}
B_4 &\leq {4(\kk+1)^4(pq)^2\over(\var(I))^2}\sum_{k,\ell=1}^{n\choose \kk+1}{\bf 1}_{\{f_k\text{ and }f_\ell\text{ share a common $(\kk-1)$-face}\}}\\
&=\cO\Big({p^2\over(n^{\kk+1}p)^2}n^{\kk+1}n\Big) = \cO(n^{-\kk})
\end{align*}
and the same bound also holds for $B_5$. As a consequence, from Theorem \ref{thm:2ndOrderPoincare} we get
$$
d_K(F,N) = \cO(p^{1/2}n^{(-\kk+1)/2}+n^{-(\kk+1)/2}p^{-1/2}+n^{-\kk/2}) = \cO(n^{-(\kk+1)/2}p^{-1/2}),
$$
and the proof of Theorem \ref{thm:Isolated faces} is complete.
\end{proof}

\begin{remark}
Similar to the estimates in Remark \ref{rem_5}, the quantity $A_3$ satisfies $A_3 = \cO\big(  n^{-\frac{\kappa+1}{2}} p^{-1/2}  \big)$, so that we can deduce from \eqref{2nd_dw} the Wasserstein bound 
\[
d_W(F, N) =\cO\big(  n^{-(\kappa+1)/2} p^{-1/2}  \big)
\]  
for the number of isolated $(\kappa-1)$-faces in the random $\kappa$-complex.
\end{remark}

\section{Proofs V: Vertices of fixed degree in hypercube percolation}\label{sec:Hypercube}

In this section we study the number of vertices $V_d$ of a fixed degree $d\in\N_0$ in the random graph ${\bf H}(n,p)$. Observe that $V_d$ can be written as a sum of random variables $I_i$, where 
\[
I_i:= \1_{\{ \text{$i$th vertex has degree $d$} \}}.
\]   
Further recall that the $n$-dimensional hypercube has exactly $2^n$ vertices and that exactly $n$ of its edges meet at each of its vertices (in geometry one says that the $n$-dimensional hypercube is a `simple' polytope). Thus, 
\[
\E[I_i]=p^d(1-p)^{n-d} \binom{n}{d}
\,\, \text{ for all $i\in\{1,\ldots,2^n\},$}
 \] and so
\[
\E[V_d] = 2^n\binom{n}{d} p^d(1-p)^{n-d}\,.
\]
To compute the variance of $V_d$, we use  
\[
\var(V_d) = \sum_{i=1}^{2^n}\var(I_i)+ \sum_{i,  j =1}^{2^n}\cov(I_i,I_j) \1_{\{ i\neq j\}} \quad{\rm and} \quad \var(I_i)=\E[I_i](1-\E[I_i]).
\]
 Moreover, if the vertices $i$ and $j$ are not adjacent in the hypercube (denoted by $i\not\sim j$), then the random variables $I_i$ and $I_j$ are independent, which implies that $\cov(I_i,I_j)=0$. Moreover, if $i$ and $j$ are neighbouring vertices in the hypercube (denoted by $i\sim j$) {and $d\geq 1$},
\begin{align*}
	\E[I_iI_j] &=\E\big[I_iI_j \, \vert \,  \text{$i\sim j$ in $\mathbf{H}(n,p)$}  \big]\PP\big[ \text{$i\sim j$ in $\mathbf{H}(n,p)$} \big] \\
	&\qquad\qquad\qquad+ \E[I_iI_j\, \vert \,  \text{$i\not\sim j$ in $\mathbf{H}(n,p)$}\big] \PP\big[  \text{$i\not\sim j$ in $\mathbf{H}(n,p)$} \big]  \\
	&= {n-1\choose d-1}^2p^{2d-1}(1-p)^{2(n-d)}+{n-1\choose d}^2p^{2d}(1-p)^{2(n-d)-1}\,,
\end{align*}
where  $\big\{ i\sim j$ in $\mathbf{H}(n,p)\big\}$ stands for the event that the neighbouring vertices $i,j$ are connected by an edge in the random graph  $\mathbf{H}(n,p)$. If $i\sim j$ and $d=0$, we easily get
\[
\E[I_iI_j] = (1-p)^{2n-1}.
\]
Therefore,  when $i\sim j$ in the hypercube, 
\begin{align*}
\cov(I_i,I_j) =\begin{cases}
{\displaystyle  \binom{n-1}{d}^2 p^{2d-1} (1-p)^{2n-2d-1} \left[ \frac{d^2(1-p)}{(n-d)^2} + p - \frac{n^2 p(1-p)}{(n-d)^2}   \right]  }& \text{if}~ d\geq 1 ;\\
  (1-p)^{2n-1}p & \text{if}~  d=0. 
  \end{cases}
\end{align*}
It follows that for $d\geq 1$,
\begin{align*}
	\var(V_d) &= 2^n \var(I_1) + \sum_{i=1}^{2^n} \sum_{j: j\sim i} \cov(I_i, I_j)   \\
	&= 2^n \binom{n}{d} p^d(1-p)^{n-d}\left(1- \binom{n}{d} p^d(1-p)^{n-d} \right)\\
	&\qquad+ 2^n n \binom{n-1}{d}^2 p^{2d-1} (1-p)^{2n-2d-1} \left[ \frac{d^2(1-p)}{(n-d)^2} + p - \frac{n^2 p(1-p)}{(n-d)^2}   \right] \\
	&= 2^n \binom{n}{d} p^d(1-p)^{n-d} +  2^n n \binom{n-1}{d}^2 p^{2d-1} (1-p)^{2n-2d-1} \left[ \frac{d^2(1-p)}{(n-d)^2} + p    \right]
\end{align*}
since the hypercube has exactly $n2^{n-1}$ edges.  For $d=0$,
\[
\var(V_0) = 2^n (1-p)^n - 2^n (1-p)^{2n} + n2^n (1-p)^{2n-1} p \geq 2^{n+1} np (1-p)^{2n-1}.
\]
When $p\to 0$ or $p\to1$ slower than exponential, we have that for any $\e\in(0,1)$, 
\begin{center}
$\var(V_0)  \geq (2-\e)^n$ for sufficiently large $n$.
\end{center}
When $p\to 0$ or $p\to1$ slower than exponential, we also have that for $d\geq 1$,
\begin{center}
$\var(V_d)  \geq (2-\e)^n$ for sufficiently large $n$.
\end{center}
We can now prove Theorem \ref{thm:IsolatedVerticesHypercube}.

\begin{proof}[Proof of Theorem \ref{thm:IsolatedVerticesHypercube}] Let   $e_1, ..., e_{n2^{n-1}}$ be arbitrarily fixed labelling of the $n2^{n-1}$ edges in the hypercube.  Define
\[
X_i := \1_{\{ \text{$e_i$ is kept in $\mathbf{H}(n,p)$}  \}} -  \1_{\{ \text{$e_i$ is removed in $\mathbf{H}(n,p)$}  \}},\qquad i\in\{1, ... , n2^{n-1}\} ,
\]
which is a sequence of independent and identically distributed Rademacher random variables with success probability $p$ (we write $q=1-p$ in what follows). The random variable $F_d\in L^\infty(\Omega)$  only depends on a finite sequence of Rademacher random variables, then all assumptions  in Theorem \ref{thm:2ndOrderPoincare} are satisfied and we just need to bound the terms $B_1,\ldots,B_5$ there. For this, we first notice that the first-order discrete gradient  and second-order discrete gradient $D_kF$, $k=1,\ldots, n2^{n-1}$, and $D_\ell D_k F$, $k,\ell=1,\ldots, n2^{n-1}$, are bounded as follows:
	\[
	|D_kF| \leq {2\sqrt{pq}\over\sqrt{\var(V_d)}} \quad {\rm and} \quad  |D_\ell D_kF| \leq {4pq\over\sqrt{\var(V_d)}}\,{\bf 1}_{\{|e_k\cap e_\ell|=1\}},
	\]
where we recall that 	$\{|e_k\cap e_\ell|=1\}$ means the edges $e_k$ and $e_\ell$ have exactly one common endpoint.  The proof is exactly the same as in the proof of Theorem \ref{thm:VertexDegrees}. Hence, using what we already have computed for $B_1,\ldots,B_5$ in the proof of this theorem and recalling that the hypercube has exactly $n2^{n-1}$ edges, we see that
	\begin{align*}
		B_1 &\leq {64(pq)^3\over(\var(V_d))^2} \sum_{j,k, \ell=1}^{n2^{n-1}} \1_{\{|e_j\cap e_\ell|= |e_k\cap e_\ell|=1\}}, & 	B_4 & \leq {64(pq)^2\over(\var(V_d))^2} \sum_{k,\ell=1}^{n2^{n-1}}{\bf 1}_{\{|e_k\cap e_\ell|=1\}},\\
		B_2 &\leq {256(pq)^3\over(\var(V_d))^2}\sum_{j,k,\ell=1}^{n2^{n-1}}\1_{\{|e_j\cap e_\ell| = |e_k\cap e_\ell|=1\}},  &  B_5 & \leq {256(pq)^2\over(\var(V_d))^2}\sum_{k,\ell=1}^{n2^{n-1}}{\bf 1}_{\{|e_k\cap e_\ell|=1\}}, \\
		B_3 & \leq {16pq\over(\var(V_d))^2}\,n2^{n-1}\,.
	\end{align*}
	Next, we notice that the sums in $B_1$ and $B_2$ are equal to $n2^{n-1}(2(n-1)(n-2)+(n-1)^2)$, since there are $n2^{n-1}$ choices for $e_\ell$, and $2(n-1)(n-2)$ possibilities to select $e_j$ and $e_k$ at the same endpoint of $e_\ell$ and $(n-1)^2$ possibilities to choose $e_j$ and $e_k$ at different endpoints of $e_\ell$. As a result, 
	\[
	\sum_{j,k, \ell=1}^{n2^{n-1}} \1_{\{|e_j\cap e_\ell|= |e_k\cap e_\ell|=1\}} \leq 3n^3 2^{n-1} = \cO\big( (2+\e)^n \big)
	\]
for any $\e\in(0,1)$. Moreover, the sum in $B_4$ and $B_5$ evaluates to 
 \[
 \sum_{k,\ell=1}^{n2^{n-1}}{\bf 1}_{\{|e_k\cap e_\ell|=1\}} =  n(n-1)2^{n} = \cO\big( (2+\e)^n \big)
 \] 
for any $\e\in(0,1)$, since there are $n2^{n-1}$ choices for $e_\ell$ and then $2(n-1)$ possibilities to select the second edge $e_k$ adjacent to $e_\ell$. 
 
 Recall that for $d\in\N_0$ and for any $\e\in(0,1)$,   $\var(V_d) \geq (2-\e)^n$ for sufficiently large $n$.  As a consequence and since $p\to 0$ or $p\to 1$ slower than exponential in $n$, we see that, for any $\e\in(0,1)$, 
	 \[
	 B_1+ \ldots +B_5 = \cO\big(  (2-\e)^{-n}\big).
	 \]
	The desired result is now a consequence of Theorem \ref{thm:2ndOrderPoincare}.
\end{proof}

\begin{remark}
Similar to the estimates in Remark \ref{rem_5},  we have for any $\e\in(0,1)$ that
\[
A_3 \leq  \frac{1}{\sqrt{pq}} \sum_{k=1}^{n2^{n-1} }  \left( \frac{2\sqrt{pq}}{\sqrt{\var(V_d)}} \right)^3 \leq {8pq\over(\var(V_d))^{3/2}}\,n2^{n-1} =\cO\big(  (2-\e)^{-n/2} \big).
\]
As a consequence, for Theorem \ref{thm:IsolatedVerticesHypercube} we have a Wasserstein bound with the same order. 
\end{remark}

\smallskip

\subsection*{Acknowledgement} {CT has been supported by the DFG priority program SPP 2265 \textit{Random Geometric Systems}.}

 \end{document}